%% file: longpaths_arxiv.tex
\documentclass[a4paper,UKenglish,cleveref, autoref, thm-restate]{lipics-v2021}
\title{Finding Matching Cuts in $H$-Free Graphs}

\author{Felicia Lucke}{Department of Informatics, University of Fribourg, Fribourg, Switzerland}{felicia.lucke@unifr.ch}{https://orcid.org/0000-0002-9860-2928}{}
\author{Daniël Paulusma}{Department of Computer Science, Durham University, Durham, UK}{daniel.paulusma@durham.ac.uk}{https://orcid.org/0000-0001-5945-9287}{}
\author{Bernard Ries}{Department of Informatics, University of Fribourg, Fribourg, Switzerland}{bernard.ries@unifr.ch}{https://orcid.org/0000-0003-4395-5547}{}

\authorrunning{F. Lucke and D. Paulusma and B.Ries} 
\Copyright{Felicia Lucke and Daniël Paulusma and Bernard Ries}

\nolinenumbers 
\hideLIPIcs

\EventEditors{John Q. Open and Joan R. Access}
\EventNoEds{2}
\EventLongTitle{42nd Conference on Very Important Topics (CVIT 2016)}
\EventShortTitle{CVIT 2016}
\EventAcronym{CVIT}
\EventYear{2016}
\EventDate{December 24--27, 2016}
\EventLocation{Little Whinging, United Kingdom}
\EventLogo{}
\SeriesVolume{42}
\ArticleNo{23}

\newcommand{\set}[1]{\ensuremath{ \left\lbrace #1 \right\rbrace }}
\usepackage{boxedminipage,tikz}
\definecolor{nicered}{RGB}{204,0,0}
\definecolor{lightblue}{RGB}{153,204,255}
\definecolor{nicegreen}{RGB}{0,153,0}
 \tikzstyle{vertex}=[thin,circle,inner sep=0.cm, minimum size=1.7mm, fill=black, draw=black]
 \tikzstyle{bvertex}=[thin,circle,inner sep=0.cm, minimum size=1.7mm, fill=lightblue, draw=lightblue]
 \tikzstyle{rvertex}=[thin,circle,inner sep=0.cm, minimum size=1.7mm, fill=nicered,draw=nicered]
 \tikzstyle{hedge}=[thick, draw = gray]
 \tikzstyle{medge}=[ultra thick, draw = black]
 \tikzstyle{gedge}=[ultra thick, draw = nicered]
  \tikzstyle{rededge}=[thick, draw = nicered]
 \tikzstyle{bluedge}=[thick, draw = lightblue]
 \tikzstyle{grnedge}=[thick, draw = nicegreen]
\newcommand{\NP}{{\sf NP}}

\newcommand{\ssi}{\subseteq_i}
\newcommand{\si}{\supseteq_i}

\newtheorem{claim1}{Claim}[theorem]

\begin{document}

\maketitle

\begin{abstract}
The well-known \NP-complete problem {\sc Matching Cut} is to decide if a graph has a matching that is also an edge cut of the graph. We prove new complexity results for {\sc Matching Cut} restricted to $H$-free graphs, that is, graphs that do not contain some fixed graph~$H$ as an induced subgraph. We also prove new complexity results for two recently studied variants of {\sc Matching Cut}, on $H$-free graphs. The first variant requires that the matching cut must be extendable to a perfect matching of the graph. The second variant requires the matching cut to be a perfect matching. In particular, we prove that there exists a small constant~$r>0$ such that the first variant is \NP-complete for $P_r$-free graphs. This addresses a question of Bouquet and Picouleau (arXiv, 2020). For all three problems, we give state-of-the-art summaries of their computational complexity for $H$-free graphs.
\end{abstract}

\section{Introduction}\label{s-intro}

Cut sets and connectivity are central topics in algorithmic graph theory.
 We consider edge cuts in graphs that have some additional structure. The common property of these cuts is that the edges in them must form a matching. Formally, consider a connected graph~$G=(V,E)$. A set $M\subseteq E$ is a {\it matching} if no two edges in $M$ have a common end-vertex. A set $M\subseteq E$ is an {\it edge cut}, if $V$ can be partitioned into sets $B$ and $R$ such that $M$ consists of all the edges with one end-vertex in $B$ and the other one in $R$. Now, $M$ is a {\it matching cut} if $M$ is a matching that is also an edge cut; see also Figure~\ref{f-examples}.
 Matching cuts are well studied due to their applications in number theory~\cite{Gr70}, graph drawing~\cite{PP01}, graph homomorphisms~\cite{GPS12}, edge labelings~\cite{ACGH12} and ILFI networks~\cite{FP82}.
The corresponding decision problem, which asks whether a given connected graph has a matching cut, is known as {\sc Matching Cut}.

We also consider two natural variants of {\sc Matching Cut}. First, let $G$ be a connected graph that has a {\it perfect} matching $M$, that is, every vertex of $G$ is incident to an edge of $M$. If $M$ contains a matching cut $M'$ of $G$, then $M$ is a {\it disconnected perfect matching} of $G$; see again Figure~\ref{f-examples} for an example. 
The problem {\sc Disconnected Perfect Matching} is to decide if a graph has a disconnected perfect matching. Every yes-instance of {\sc Disconnected Perfect Matching} is a yes-instance of {\sc Matching Cut}, but the reverse might not be true; for example, the $3$-vertex path has a matching cut but no (disconnected) perfect matching.

Suppose now that we search for a matching cut with a maximum number of edges, or for a disconnected perfect matching with a matching cut that is as large as possible. In both settings, the extreme case is when the matching cut is a perfect matching itself. Such a matching cut is called {\it perfect}; see Figure~\ref{f-examples}. By definition, a perfect matching cut is a disconnected perfect matching, but the reverse might not hold: take the cycle on six vertices which has several disconnected perfect matchings but no perfect matching cut.
The problem {\sc Perfect Matching Cut} is to decide if a connected graph has a perfect matching cut.

\begin{figure}[ht]
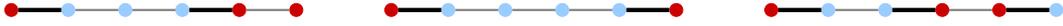

\centering
\include{intro-variants2}
\caption{\label{fig-intro-variants}The graph $P_6$ with a matching cut that is not contained in a disconnected perfect matching (left), a matching cut that is properly contained in a disconnected perfect matching (middle) and a perfect matching cut (right). In each figure, thick edges denote matching cut edges.}\label{f-examples}
\end{figure}

\noindent
All three problems are known to be \NP-complete, as we will explain in more detail below. Hence, it is natural to restrict the input to some special graph class to obtain a better understanding of the computational hardness of some problem, or some set of problems. In particular, jumps in complexity can be large and unexpected. To give an extreme example~\cite{MPS22}, there exist problems that are PSPACE-complete in general but constant-time solvable for every other {\it hereditary} graph class, i.e., that is closed under vertex deletion.
 
It is readily seen that a graph class is hereditary if and only if it can be characterized by a set of forbidden induced subgraphs. A well-known example of a family of hereditary graph classes is obtained when we forbid a single subgraph $H$. That is, a graph~$G$ is {\it $H$-free} if $G$ does not contain $H$ as induced subgraph, or equivalently, if $G$ cannot be modified into~$H$ by a sequence of vertex deletions. The class of $H$-free graphs has proven to be an ideal testbed for a systematic study into the complexity of many classical graph problems and graph parameters, as can not only be seen from surveys for e.g. {\sc Colouring}~\cite{GJPS17,RS04} or clique-width~\cite{DJP19}, but also from extensive studies for specific $H$-free graphs, e.g. bull-free graphs~\cite{Ch12} or claw-free graphs~\cite{CS05,HMLW11}.  As such, we will also focus on $H$-free graphs in this paper. Before presenting our results we first discuss relevant known results.  

\subsection{Known Results}

Out of the three problems, {\sc Matching Cut} has been studied most extensively. Already in the eighties, Chvátal~\cite{Ch84} proved that \textsc{Matching Cut} is \NP-complete. Afterwards a large number of complexity results were proven for special graph classes. Here, we only discuss those results that are relevant for our context, whereas results for non-hereditary graph classes can, for example, be found in~\cite{BJ08,LL19}. In particular, we refer to a recent paper of Chen et al.~\cite{CHLLP21} for a comprehensive overview.

On the positive side, Bonsma~\cite{Bo09} proved that {\sc Matching Cut} is polynomial-time solvable for $K_{1,3}$-free graphs and $P_4$-free graphs.  Recently, Feghali~\cite{Fe23} proved the same for $P_5$-free graphs, which we extended to $P_6$-free graphs in~\cite{LPR}. In the latter paper, we also showed that if {\sc Matching Cut} is polynomial-time solvable for $H$-free graphs, for some graph $H$, then it is so for $(H+P_3)$-free graphs (see Section~\ref{s-pre} for any unexplained notation and terminology).

On the negative side, {\sc Matching Cut} is \NP-complete even for $K_{1,4}$-free graphs. This follows from the construction of Chvátal~\cite{Ch84} (see also~\cite{Bo09,KL16}). Bonsma~\cite{Bo09} proved that {\sc Matching Cut} is \NP-complete for planar graphs of girth~$5$, and thus for $C_r$-free graphs with $r\in \{3,4\}$. Le and Randerath~\cite{LR03} proved that {\sc Matching Cut} is \NP-complete for $K_{1,5}$-free bipartite graphs. Hence, it is \NP-complete for $H$-free graphs if $H$ has an odd cycle. Via a trick of Moshi~\cite{Mo89}, \NP-completeness for $H$-free graphs also holds if $H$ has an even cycle (see~\cite{LPR}). Feghali~\cite{Fe23} proved the existence of an unspecified constant $r$ such that {\sc Matching Cut} is \NP-complete for $P_r$-free graphs; we will show that $r=27$ in his~construction.  

We now turn to \textsc{Disconnected Perfect Matching}. This problem was introduced by Bouquet and Picouleau~\cite{BP}, under a different name, but to avoid confusion with {\sc Perfect Matching Cut}, Le and Telle~\cite{LT21} introduced the notion of disconnected perfect matchings, which we adapted. As observed in~\cite{BP}, for cubic graphs, the problem is equivalent to finding a disconnected $2$-factor. Hence, it follows from a result of Diwan \cite{Di00} that every planar cubic bridgeless graph, except the $K_4$, has a disconnected perfect matching. Bouquet and Picouleau~\cite{BP} proved that {\sc Disconnected Perfect Matching} is, among others, polynomial-time solvable for claw-free graphs and $P_5$-free graphs, but \NP-complete for bipartite graphs (of diameter~$4$), for $K_{1,4}$-free planar graphs (each vertex of which has either degree~$3$ or~$4$) and for planar graphs with girth~$5$.

Finally, we discuss {\sc Perfect Matching Cut}. Heggernes and Telle~\cite{HT98} proved that this problem is \NP-complete.
 Le and Telle~\cite{LT21} proved that for every integer~$g\geq 3$, {\sc Perfect Matching Cut} is \NP-complete even for $K_{1,4}$-free bipartite graphs of girth $g$. The same authors showed that the problem is polynomial-time solvable for the classes of $S_{1,2,2}$-free graphs (which contain the classes of $K_{1,3}$-free graphs and $P_5$-free graphs) and for chordal graphs. As explained in~\cite{LT21}, the latter result generalizes a known result for interval graphs, for which a branch decomposition  of constant mim-width can be computed in polynomial~time.
  
\subsection{New Results}

For {\sc Matching Cut} on $H$-free graphs, the remaining cases are when $H$ is a $P_{27}$-free forest, each vertex of which has degree at most~$3$, such that $H$ is not an induced subgraph of $P_6+sP_3$ or $K_{1,3}+sP_3$ for some constant $s\geq 0$. By modifying the construction of Feghali~\cite{Fe23}, 
we prove in Section~\ref{s-main1} that {\sc Matching Cut} is \NP-complete for $(4P_5,P_{19})$-free graphs.  Using the aforementioned trick of Moshi~\cite{Mo89}, we also observe that {\sc Matching Cut} is \NP-complete for $H^*$-free graphs, where $H^*$ is the graph that looks like the letter $H$. 

For {\sc Disconnected Perfect Matching} on $H$-free graphs, the remaining cases are when $H$ contains an even cycle of length at least~$6$, such that every vertex of $H$ has degree at most~$3$ and $H$ is not an induced subgraph of $K_{1,3}$ or $P_5$. Bouquet and Picouleau~\cite{BP} asked about the complexity of the problem for $P_r$-free graphs, with $r\geq 6$. We partially answer their question by proving  \NP-completeness for $(4P_7,P_{23})$-free graphs in Section~\ref{s-main1} (via modifying our construction for {\sc Matching Cut} for $(4P_5,P_{19})$-free graphs).

For {\sc Perfect Matching Cut} on $H$-free graphs, the remaining cases are when $H$ is a forest of maximum degree~$3$, such that $H$ is not an induced subgraph of $S_{1,2,2}$. In Section~\ref{s-main3}, we first prove that {\sc Perfect Matching Cut} is polynomial-time solvable for graphs of radius at most~$2$, and we use this result to obtain a polynomial-time algorithm for $P_6$-free graphs. We also prove that if {\sc Perfect Matching Cut} is polynomial-time solvable for $H$-free graphs, for some graph $H$, then it is so for $(H+P_4)$-free graphs. All our results are obtained by combining a number of known propagation rules~\cite{LL19,LT21} with new rules that we will introduce. After applying these rules exhaustively, we obtain a graph, parts of which have been allocated to the sides $B$ and $R$ of the edge cut that we are looking for. We will prove that the connected components of the remaining subgraph will be placed completely in $B$ or $R$, and that this property suffices. By doing so, we extend a known approach with our new rules and show that in this way we widen its applicability.

The following three theorems present the state-of-art for $H$-free graphs; here  we write $G'\ssi G$ to indicate that $G'$ is an induced subgraph of $G$; as mentioned, recall that all undefined notation can be found in Section~\ref{s-pre}. 

\begin{theorem}\label{t-main1}
For a graph~$H$, {\sc Matching Cut} on $H$-free graphs is 
\begin{itemize}
\item polynomial-time solvable if $H\ssi sP_3+K_{1,3}$ or $sP_3+P_6$ for some $s\geq 0$, and
\item \NP-complete if $H\si C_r$ for some $r\geq 3$, $K_{1,4}$, $P_{19}$, $4P_5$ or $H^*$.
\end{itemize}
\end{theorem}

\begin{theorem}\label{t-main2}
For a graph~$H$, {\sc Disconnected Perfect Matching} on $H$-free graphs is 
\begin{itemize}
\item polynomial-time solvable if $H\ssi K_{1,3}$ or $P_5$, and
\item \NP-complete if $H\si C_r$ for some odd $r\geq 3$, $C_4$, $K_{1,4}$, $P_{23}$ or $4P_7$.
\end{itemize}
\end{theorem}

\begin{theorem}\label{t-main3}
For a graph~$H$, {\sc Perfect Matching Cut} on $H$-free graphs is 
\begin{itemize}
\item polynomial-time solvable if $H\ssi sP_4+S_{1,2,2}$ or $sP_4+P_6$, for some $s\geq 0$, and
\item \NP-complete if $H\si C_r$ for some $r\geq 3$ or $K_{1,4}$.
\end{itemize}
\end{theorem}

\noindent
We state a number of open problems that originate from our systematic study in Section~\ref{s-con}.

\section{Preliminaries}\label{s-pre}

We only consider finite undirected graphs without multiple edges and self-loops. 
Throughout this section, we let $G=(V,E)$ be a connected graph. Let $u\in V$. The set $N(u)=\{v \in V\; |\; uv\in E\}$ is the {\it neighbourhood} of $u$ in $G$, where $|N(u)|$ is the {\it degree} of $u$. Let $S\subseteq V$. The {\it neighbourhood} of $S$ is the set $N(S)=\bigcup_{u\in S}N(u)\setminus S$. 
The graph $G[S]$ is the subgraph of $G$ {\it induced} by $S\subseteq V$, that is, $G[S]$ is the graph obtained from $G$ after deleting the vertices not in $S$. We write $G'\ssi G$, if $G'$ is an induced subgraph of $G$. We say that $S$ is a {\it dominating} set of $G$, and that $G[S]$ {\it dominates} $G$, if every vertex of $V\setminus S$ has at least one neighbour in~$S$. The {\it domination number} of $G$ is the size of a smallest dominating set of $G$.

Let $u,v\in V$.
The {\it distance} between $u$ and $v$ in~$G$ is the {\it length} (number of edges) of a shortest path between $u$ and $v$ in $G$. The {\it eccentricity} of $u$ is the maximum distance between $u$ and any other vertex of $G$. 
The {\it radius} of $G$ is the minimum eccentricity over all vertices of~$G$. 
If $G$ is not a tree, then the {\it girth} of $G$ is the length of a shortest cycle in $G$.

Let $H$ be a graph. Recall that $G$ is $H$-free if $G$ does not contain $H$ as an induced subgraph. Let $\{H_1,\ldots,H_n\}$ be a set of graphs. Then $G$ is {\it $(H_1,\ldots,H_n)$-free}, if $G$ is $H_i$-free for every $i\in \{1,\ldots,n\}$.
The graph $P_r$ is the path on $r$ vertices. The graph $C_r$ is the cycle on $r$ vertices.
A bipartite graph with non-empty partition classes $V_1$ and $V_2$ is {\it complete} if there is an edge between 
every vertex of $V_1$ and every vertex of $V_2$. If $|V_1|=k$ and $|V_2|=\ell$, we write $K_{k,\ell}$.
The graph $K_{1,\ell}$ is the {\it star} on $\ell+1$ vertices. The graph $K_{1,3}$ is also known as the {\it claw}.
For $1\leq h\leq i\leq j$, the graph $S_{h,i,j}$ is the tree with one vertex of degree~$3$,
whose (three) leaves are at distance~$h$,~$i$ and~$j$ from the vertex of degree~$3$. Observe that $S_{1,1,1}=K_{1,3}$.
We need the following known result (which has been strengthened in~\cite{CS16}).

\begin{theorem}[\cite{HP10}]\label{t-hp}
A graph $G$ is $P_6$-free if and only if each connected induced subgraph of $G$ 
contains a dominating induced $C_6$ or 
a dominating (not necessarily induced) complete bipartite graph.
Moreover, such a dominating subgraph of $G$ can be found in polynomial time.
\end{theorem}

\noindent
Let $G_1$ and $G_2$ be two vertex disjoint graphs. The graph  $G_1+G_2=(V(G_1)\cup V(G_2),E(G_1)\cup E(G_2))$ is
the {\it disjoint union} of $G_1$ and $G_2$. For a graph $G$, the graph $sG$ is the disjoint union of $s$ copies of $G$.
Let $H^*$ be the ``H''-graph, which is the graph on six vertices obtained from the $2P_3$ by adding an edge joining the middle vertices of the two $P_3$s. 

A {\it red-blue colouring} of $G$ colours every vertex of $G$ either red or blue. If every vertex of a set $S\subseteq V$ has the same colour (red or blue), then $S$ (and also $G[S]$) are called {\it monochromatic}.
A red-blue colouring is {\it valid}, if every blue vertex has at most one red neighbour; every red vertex has at most one blue neighbour; and both colours red and blue are used at least once.  If a red vertex $u$ has a blue vertex neighbour $v$, then $u$ and $v$ are \textit{matched}.
See also Figure~\ref{f-examples}.

For a valid red-blue colouring of $G$, we let $R$ be the {\it red} set consisting of all vertices coloured red and $B$ be the {\it blue} set consisting of all vertices coloured blue (so $V=R\cup B$). Moreover, the {\it red interface} is the set $R'\subseteq R$ consisting of all vertices in $R$ with a (unique) blue neighbour, and the {\it blue interface} is the set $B'\subseteq B$ consisting of all vertices in $B$ with a (unique) red neighbour in $R$. 
A red-blue colouring of~$G$ is {\it perfect}, if it is valid and moreover $R'=R$ and $B'=B$. A red-blue colouring of a graph~$G$ is {\it perfect-extendable}, if it is valid and $G[R\setminus R']$ and $G[B \setminus B']$ both contain a perfect matching. In other words, the matching given by the valid red-blue colouring can be extended to a perfect matching in $G$ or, equivalently, is contained in a perfect matching in $G$.
We can now make the following known observation.

\begin{observation}\label{o} Let $G$ be a connected graph. The following three statements hold:
\begin{itemize}
\item [(i)] $G$ has a matching cut if and only if $G$ has a valid red-blue colouring;
\item [(ii)] $G$ has a disconnected perfect matching if and only if $G$ has a perfect-extendable red-blue colouring;
\item [(iii)] $G$ has a perfect matching cut if and only if $G$ has a perfect red-blue colouring.
\end{itemize}
\end{observation}

\section{Our NP-Completeness Results}\label{s-main1}

\noindent
We prove three \NP-completeness results in this section. Our first result is a straightforward observation.
Let $uv$ be an edge in a graph $G$. Replacing $uv$ by new vertices $w_1$ and $w_2$ and edges $uw_1$, $uw_2$, $vw_1$, $vw_2$ is a {\it $K_{2,2}$-replacement}. Let $G_{uv}$ be the new graph; see also Figure~\ref{f-moshi}. Moshi~\cite{Mo89} showed that~$G$ has a matching cut if and only if $G_{uv}$ has a matching cut. Applying a $K_{2,2}$-replacement on every edge to ensure that no two degree-$3$ vertices are adjacent anymore leads to the~following:

\begin{figure}
\begin{center}
\scalebox{1.5}{\input{Moshi.tex}}
\caption{The $K_{2,2}$-replacement applied on edge $uv$.}\label{f-moshi} 
\end{center}
\end{figure}
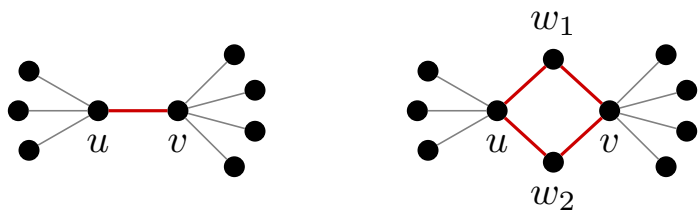

\begin{theorem}\label{thm:h*}
{\sc Matching Cut} is \NP-complete for $H^*$-free graphs.
\end{theorem}

\noindent
For proving our next two \NP-completeness results, we reduce from 
The {\sc Exact Positive 1-in-3 SAT}. This problem
 takes as input a pair $(X,C)$, where $X$ is a set 
of variables and $C$ is a set of clauses, each containing 
exactly three literals, all three of which are positive. Moreover, each variable of $X$ appears in exactly three clauses of $C$. The question is whether there exists a truth assignment, such that each clause contains exactly one true literal. 

\begin{theorem}[\cite{Sc10}]\label{t-sc10}
{\sc Exact Positive 1-in-3 SAT} is \NP-complete.
\end{theorem}

\noindent
Theorem~\ref{t-p19} is our first new result. 
Its proof follows from Feghali's construction~\cite{Fe23} after making some minor modifications to it.
For completeness, and since we use the modified construction as a basis for the proof of Theorem \ref{t-p23}, we added a detailed proof. Recall that Feghali~\cite{Fe23} showed that {\sc Matching Cut} is \NP-complete for $P_r$-free graphs, for some unspecified constant~$r$. We will show that
in~\cite{Fe23} $r=27$; see Remark~\ref{remark-p27} below.

\begin{theorem}\label{t-p19}
{\sc Matching Cut} is \NP-complete for $(4P_5,P_{19})$-free graphs.
\end{theorem}

\begin{proof}
{\sc Matching Cut} is  in \NP, since it is possible to check in polynomial time if a given red-blue-colouring is valid or not. To prove \NP-hardness, we will use a reduction from {\sc Exact Positive 1-in-3 SAT}, which is \NP-complete by Theorem \ref{t-sc10}. Let $\mathcal{I}$ be an instance of {\sc Exact Positive 1-in-3 SAT} with variable set~$X$ and clause set $C$. We will build a graph $G_{\mathcal{I}}$ (see also Figures~\ref{fig-p17-const1} and \ref{fig-p17-const2}):

\begin{figure}[ht]
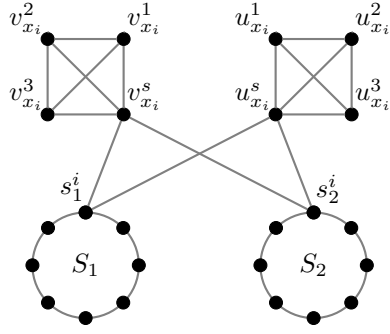

\centering
\include{p17-const1}
\caption{\label{fig-p17-const1}The cliques $S_1$ and $S_2$ together with the variable gadget of $x_i$.}
\end{figure}

\begin{itemize}
\item for every $x_i \in X$, construct a variable gadget consisting of two disjoint cliques of size~$4$, $U_{x_i}$ and $V_{x_i}$, with vertex set $\set{u^s_{x_i}, u^1_{x_i},u^2_{x_i},u^3_{x_i}}$ and $\set{v^s_{x_i}, v^1_{x_i}, v^2_{x_i}, v^3_{x_i}}$, respectively;
\item add two cliques $S_1$ and $S_2$ with vertex set $\{s_1^1,\ldots,s_1^{|X|}\}$ and $\{s_2^1,\ldots,s_2^{|X|}\}$, respectively;
\item for every $i\in \{1,\ldots,|X|\}$, add the edges $s_1^iu^s_{x_i},s_1^iv^s_{x_i},s_2^iu^s_{x_i},s_2^iv^s_{x_i}$;
\item for every $c_j\in C$, construct a clause gadget on clause vertices $v_{c_j}, u^1_{c_j},u^2_{c_j}$, and auxiliary vertices $a^1_{c_j}, \ldots, a^{6}_{c_j}$. \item Add edges between all clause vertices of all clause gadgets to obtain a clique $S_3$.
\item for every $j\in \{1,\ldots,|C|\}$, add the edges $u^1_{c_j}a^\ell_{c_j}$, for $\ell=1,2,3$ and $u^2_{c_j}a^{\ell}_{c_j}$, for $\ell=4,5,6$;
\item for every $x_i\in X$ occurring in clauses $c_{j_1},c_{j_2}$ and $c_{j_3}$, add the edges $v^k_{x_i}v_{c_{j_k}}$, for $k=1,2,3$; and
\item for every $c_j\in C$ such that $c_j=x_{i_1}\vee x_{i_2}\vee x_{i_3}$, add the edges  $u^k_{x_{i_k}}a^k_{c_{j}}$ and $u^k_{x_{i_k}}a^{k+3}_{c_{j}}$, for $k=1,2,3$. 
\end{itemize}

\begin{figure}[ht]
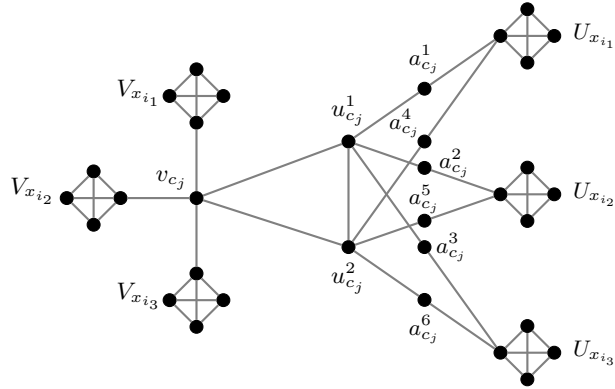

\centering
\include{p17-const2}
\caption{\label{fig-p17-const2}The clause gadget for clause $c_j=x_{i_1}\vee x_{i_2}\vee x_{i_3}$.}
\end{figure}

We claim that $\mathcal{I}$ admits a truth assignment such that each clause contains exactly one true literal if and only if $G_{\mathcal{I}}$ admits a valid red-blue-colouring. 

\medskip
\noindent
First suppose that $G_{\mathcal{I}}$ admits a valid red-blue-colouring.
We start with some useful claims. 

\begin{claim1}\label{c-p17-1}
For any variable $x_i\in X$, $i=1,\ldots,|X|$, both $V_{x_i}$ and $U_{x_i}$ are monochromatic. Furthermore, $S_1$, $S_2$ and $S_3$ are each monochromatic.
\end{claim1}
\begin{claimproof}
This immediately follows from the fact these sets are cliques of size at least~$3$.
\end{claimproof}

\noindent
We say that a monochromatic set has {\it colour} red or blue if all its vertices are coloured red or blue, respectively.

\begin{claim1}\label{c-p17-2}
It holds that
$S_1$ and $S_2$ have different colours.
\end{claim1}
\begin{claimproof}
Suppose for a contradiction that 
$S_1$ and $S_2$ have the same colour. We may assume, without loss of generality, that $S_1$ and $S_2$ are both coloured blue. Since for every variable $x_i\in X$, $i=1,\ldots,|X|$, there exist vertices $u^s_{x_i}$ and $v^s_{x_i}$ having each a neighbour in both $S_1$ and $S_2$, it follows that every variable gadget is coloured blue. This implies in particular that every vertex $v_{c_j}$, for $j=1,\ldots,|C|$, has three blue neighbours and hence, is coloured blue itself. Further, since $S_3$ is monochromatic by Claim~\ref{c-p17-1}, all the vertices $u^1_{c_j},u^2_{c_j}$, for $j=1,\ldots,|C|$, are coloured blue. Thus, both neighbours of each auxiliary vertex are blue, which forces the auxiliary vertices to be blue themselves. It follows that all vertices in $G_{\mathcal{I}}$ are coloured blue. Hence, the colouring is not valid, a contradiction.
\end{claimproof}

\begin{claim1}\label{c-p17-3}
For every variable $x_i\in X$, $i\in\{1,\ldots,|X|\}$, $U_{x_i}$ and $V_{x_i}$ have different colours.
\end{claim1}
\begin{claimproof}
Suppose for a contradiction that for some variable $x_i\in X$, $i\in\{1,\ldots,|X|\}$,
$U_{x_i}$ and $V_{x_i}$ have the same colour. We may assume without loss of generality that they are both coloured blue. Since $s_1^i$ and $s_2^i$ are both adjacent to $u^s_{x_i}$ and to $v^s_{x_i}$, it follows that they are both coloured blue. Now it follows from Claim~\ref{c-p17-1}, that $S_1$ and $S_2$ must both be coloured blue, a contradiction to Claim~\ref{c-p17-2}.
\end{claimproof}

\begin{claim1}\label{c-p17-4}
For every clause $c_j \in C$, exactly two neighbours of $v_{c_j}$ have the same colour as~$v_{c_j}$.
\end{claim1}
\begin{claimproof}
We may assume without loss of generality that $v_{c_j}$ is coloured blue.
Let $ c_j = (x_{i_1} \vee x_{i_2} \vee x_{i_3})$. Let $v^1_{x_{i_1}},v^2_{x_{i_2}},v^3_{x_{i_3}}$ be the three neighbours of $v_{c_j}$ outside of $S_3$ and let $u^1_{x_{i_1}}, u^2_{x_{i_2}}$ and $u^3_{x_{i_3}}$ be the neighbours of the auxiliary vertices $a^k_{c_j}$, $k=1,\ldots,6$. By definition of a valid red-blue-colouring, $v_{c_j}$ has at least two neighbours outside of $S_3$ that are coloured blue. Suppose for a contradiction that the vertices $v^1_{x_{i_1}},v^2_{x_{i_2}},v^3_{x_{i_3}}$ are all coloured blue. Then, it follows from Claim~\ref{c-p17-1}, that $V_{x_{i_1}}, V_{x_{i_2}}, V_{x_{i_3}}$  must be coloured blue. Notice that, since $v_{c_j}$ is coloured blue, Claim~\ref{c-p17-1} also implies that all vertices in $S_3$, and in particular $u^1_{c_j}$ and $u^2_{c_j}$, are coloured blue.  Let $A_1$ (resp. $A_2$) be the set of auxiliary vertices which are adjacent to $u^1_{c_j}$ (resp. $u^2_{c_j}$). Then, at least two vertices in $A_1$ and two vertices in $A_2$ are coloured blue. Since $u^1_{x_{i_1}}, u^2_{x_{i_2}}$ and $u^3_{x_{i_3}}$ have each one neighbour in $A_1$ and one neighbour in $A_2$, it follows that one of them has two blue neighbours in $A_1\cup A_2$, and is therefore coloured blue. We may assume, without loss of generality, that this vertex is $u^1_{x_{i_1}}$. Using Claim~\ref{c-p17-1} again, we get that $U_{x_{i_1}}$ is coloured blue, a contradiction to Claim~\ref{c-p17-3}.
\end{claimproof}

\noindent 
We continue as follows.
By Claim~\ref{c-p17-1}, 
we may assume without loss of generality that $S_3$ is coloured blue. Then, we set every variable $x_i\in X$, $i\in \{1,\ldots,|X|\}$, to true for which $V_{x_i}$ has been coloured red. We set all other variables to false. By Claim~\ref{c-p17-4}, we know that for each clause $c_j \in C$, $j\in\{1,\ldots,|C|\}$, there exists exactly one red neighbour of $v_{c_j}$. Hence, in every clause, exactly one literal is set to true. Since by Claim~\ref{c-p17-1}, every $V_{x_i}$, $i\in \{1,\ldots,|X|\}$, is monochromatic, it follows that no variable gets both values true and false. Thus, $\mathcal{I}$ admits a truth assignment such that each clause contains exactly one true literal.\\

\begin{figure}[ht]
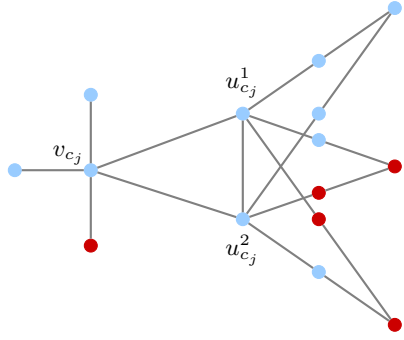

\centering
\include{p17-colouring}
\caption{\label{fig-p17-colouring}The given valid red-blue-colouring of the clause gadget for $c_j$.}
\end{figure}

\noindent
Now suppose that $\mathcal{I}$ admits a truth assignment such that each clause contains exactly one true literal. For every variable $x_i\in X$, $i\in\{1,\ldots,|X|\}$, that is set to true, we colour $V_{x_i}$ red and $U_{x_i}$ blue; for every other variable $x_i\in X$, we colour $V_{x_i}$ blue and $U_{x_i}$ red. It follows that every vertex $v_{c_j}$, for $c_j\in C$ and $j\in\{1,\ldots,|C|\}$, has exactly one red and two blue neighbours outside of $S_3$, since the truth assignment is such that each clause contains exactly one true literal. Thus, we colour $S_3$ blue. 
If we consider a clause $c_j = (x_{i_1} \vee x_{i_2} \vee x_{i_3})$, we know that exactly one of $U_{x_{i_1}}, U_{x_{i_2}}, U_{x_{i_3}}$ is blue and the other ones are red.  Assume, without loss of generality, that $U_{x_{i_1}}$ is blue and consider $a^1_{c_{j}}, a^4_{c_{j}}$, the neighbours of  $u^1_{x_{i_1}}\in U_{x_{i_1}}$. We then colour $a^1_{c_{j}}, a^4_{c_{j}}$ blue, since they have two blue neighbours ($u^1_{x_{i_1}}$ and $u^1_{c_j}$ resp. $u^2_{c_j}$).  To obtain a valid red-blue colouring of $G_{\mathcal{I}}$, we have to colour one of the vertices $a^2_{c_{j}}, a^5_{c_{j}}$ blue and the other one red, and similarly, one of the vertices $a^3_{c_{j}}, a^6_{c_{j}}$ blue and the other one red. Since $u^1_{c_j},u^2_{c_j}$ are both coloured blue and each of them can have at most one red neighbour, we colour $a^2_{c_{j}}, a^6_{c_{j}}$ blue and $a^3_{c_{j}}, a^5_{c_{j}}$ red (see Figure \ref{fig-p17-colouring}). Finally, the only vertices that remain uncoloured are the vertices in $S_1$ and $S_2$. Here, the only restriction is that $S_1$ and $S_2$ are coloured differently, hence we colour for instance $S_1$ blue and $S_2$ red. This clearly gives us a valid red-blue-colouring of $G_{\mathcal{I}}$.

\medskip
\noindent
To complete the proof, it remains to show that $G_{\mathcal{I}}$ is $(4P_5,P_{19})$-free. Let $P$ be a longest induced path in $G_{\mathcal{I}}$. $P$ can contain at most two vertices from a same clique, since otherwise it would not be induced. Also, if $P$ contains two vertices from a same clique, then these vertices are necessarily consecutive in $P$. Let $W_1,W_2 \in \set{U_{x_i},V_{x_i} |\ x_i \in X}$. By construction, every path from a vertex in $W_1$ to a vertex in $W_2$ contains at least one vertex from one of the cliques $S_1,S_2,S_3$. Hence, $P$ can intersect at most 4 cliques belonging to some variable gadgets; further, if it does intersect 4 cliques belonging to some variable gadgets, then it intersects each of the cliques $S_1,S_2,S_3$ as well. Notice that after respectively before intersecting $S_3$, $P$ may contain at most one auxiliary vertex (see Figure \ref{fig-p17-path} for an example). Finally, the end-vertices of $P$ may correspond each to an auxiliary vertex. So $P$ contains at most 18 vertices, and thus, $G_{\mathcal{I}}$ is indeed $P_{19}$-free.

\begin{figure}[ht]
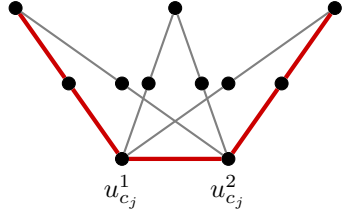

\centering
\include{p17-path}
\caption{\label{fig-p17-path}A path intersecting a clause gadget and containing two auxiliary vertices.}
\end{figure}

Let $P$ now be a $P_5$ in $G_\mathcal{I}$.  Assume that $P$ does not contain any vertex of the cliques $S_1,S_2,S_3$. Then, $P$ can only contain vertices from variable gadgets and auxiliary vertices. As mentioned above, $P$ can contain at most two vertices from a same clique, and further every path from a vertex in $W_1$ to a vertex in $W_2$, where $W_1,W_2 \in \set{U_{x_i},V_{x_i} |\ x_i \in X}$, contains at least one vertex from one of the cliques $S_1,S_2,S_3$. Hence, $P$ has length at most 3, a contradiction. We conclude that $P$ must intersect at least one of the cliques $S_1$, $S_2$, $S_3$. It follows that there can exist at most 3 pairwise induced paths of length 4, so $G_\mathcal{I}$ is $4P_5$-free.
\end{proof}

\begin{remark}\label{remark-p19}
It can be verified that the graph $G_{\mathcal{I}}$  in the proof Theorem~\ref{t-p19} is not $P_{18}$-free and not $sP_4$-free for any $s\geq 1$.
\end{remark} 

\begin{remark}\label{remark-p27}
In the graph $G_{\mathcal{I}}$ from the proof of Theorem~\ref{t-p19}, no vertex of~$U_{x_i}$ is adjacent to a vertex of $V_{x_i}$, for any $x_i \in X$. In Feghali's construction~\cite{Fe23}, there is an edge between any two such cliques. This implies that an induced path can use four consecutive vertices inside the same variable gadget. Via similar arguments as in our proof, one can easily show that Feghali's construction has an induced $P_{26}$, but no induced $P_{27}$ (and thus it is $P_{27}$-free).
\end{remark}

\medskip
\noindent
We now modify the construction in the proof of Theorem~\ref{t-p19} to obtain the following result for {\sc Disconnected Perfect Matching}, which addresses a question of Bouquet and Picouleau~\cite{BP}.

\begin{theorem}\label{t-p23}
{\sc Disconnected Perfect Matching} is \NP-complete for $(4P_7,P_{23})$-free graphs.
\end{theorem}

\begin{proof}
We first note that {\sc Disconnected Perfect Matching} belong to \NP, as we can verify in polynomial time if a given  perfect-extendable red-blue-colouring is valid or not.

In order to prove \NP-hardness, we reduce from {\sc Exact Positive 1-in-3 SAT}, which is \NP-complete by Theorem \ref{t-sc10}. 
Let $\mathcal{I}$ be an instance of 
 {\sc Exact Positive 1-in-3 SAT} 
with variable set~$X$ and clause set $C$. We build a graph $G_{\mathcal{I}}$ similarly to the graph in Theorem~\ref{t-p23} (see also Figures~\ref{fig-p23-const1} and~\ref{fig-p23-const2}): 

\begin{itemize}
\item for every $x_i \in X$, construct a variable gadget consisting of two disjoint cliques of size~$7$, $U_{x_i}$ and $V_{x_i}$, with vertex set $\set{u^s_{x_i}, u^1_{x_i},\ldots,u^6_{x_i}}$ and $\set{v^s_{x_i}, v^1_{x_i}, \ldots, v^6_{x_i}}$, respectively;
\item add two cliques $S_1$ and $S_2$ with vertex set $\{s_1^1,\ldots,s_1^{|X|}\}$ and $\{s_2^1,\ldots,s_2^{|X|}\}$, respectively;
\item for every $i\in \{1,\ldots,|X|\}$, add the edges $s_1^iu^s_{x_i},s_1^iv^s_{x_i},s_2^iu^s_{x_i},s_2^iv^s_{x_i}$;
\item for every $c_j\in C$, construct a clause gadget on clause vertices $v^1_{c_j},v^2_{c_j}, u^1_{c_j},u^2_{c_j},u^3_{c_j},u^4_{c_j}$, and auxiliary vertices $a^1_{c_j}, \ldots, a^{12}_{c_j}$. 
\item Add edges between all clause vertices of all clause gadgets to obtain a clique $S_3$.
\item for every $j\in \{1,\ldots,|C|\}$, add the edges $u^1_{c_j}a^\ell_{c_j}$, for $\ell=1,2,3$, $u^2_{c_j}a^{\ell}_{c_j}$, for $\ell=4,5,6$, $u^3_{c_j}a^{\ell}_{c_j}$, for $\ell=7,8,9$ and $u^4_{c_j}a^{\ell}_{c_j}$, for $\ell=10,11,12$. Add also the edges $a_{c_j}^\ell a_{c_j}^{\ell+6}$, for $\ell=1,\dots, 6$;
\item for every $x_i\in X$ occurring in clauses $c_{j_1},c_{j_2}$ and $c_{j_3}$, add the edges $v^k_{x_i}v^1_{c_{j_k}}$ and $v^{k+3}_{x_i}v^2_{c_{j_k}}$, for $k=1,2,3$; and
\item for every $c_j\in C$ such that $c_j=x_{i_1}\vee x_{i_2}\vee x_{i_3}$, add the edges  $u^k_{x_{i_k}}a^k_{c_{j}},u^k_{x_{i_k}}a^{k+3}_{c_{j}},u^{k+3}_{x_{i_k}}a^{k+6}_{c_{j}}$ and $u^{k+3}_{x_{i_k}}a^{k+9}_{c_{j}}$, for $k=1,2,3$. 
\end{itemize}

\begin{figure}[ht]
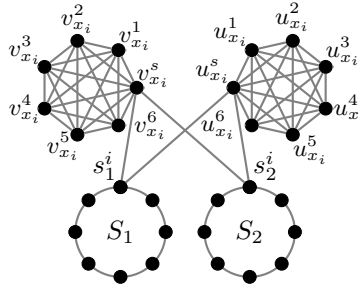

\centering
\include{p23-const1}
\caption{\label{fig-p23-const1}The cliques $S_1$ and $S_2$, together with the variable gadget of $x_i$.}
\end{figure}

\begin{figure}[ht]
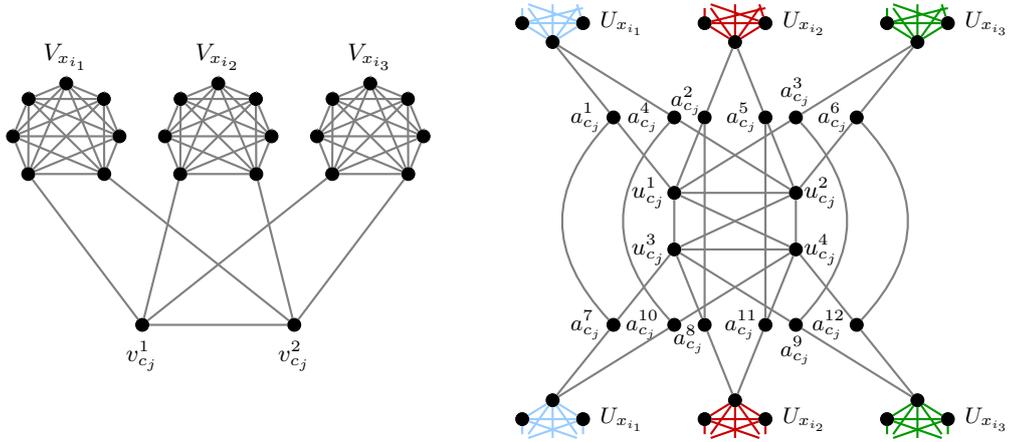

\centering
\include{p23-const2}
\caption{\label{fig-p23-const2}The clause gadget for the clause $c_j=x_{i_1}\vee x_{i_2}\vee x_{i_3}$. The vertices $v_{c_j}^1,v_{c_j}^2,u^1_{c_j},\ldots,u^4_{c_j}$ belong to clique $S_3$; for readability, we omitted the edges between $v_{c_j}^1,v_{c_j}^2$ and $u^1_{c_j},\ldots,u^4_{c_j}$. Moreover, coloured edges of the same colour belong to the same clique $U_{x_{i_j}}$, for $j=1,2,3$ (again, for readability).} 
\end{figure}

\noindent
We claim that $\mathcal{I}$ has a truth assignment such that each clause contains exactly one true literal if and only if $G_{\mathcal{I}}$ has a perfect-extendable red-blue-colouring. 

First suppose that $G_{\mathcal{I}}$ admits a valid perfect-extendable red-blue-colouring.
We start with some useful claims, the first three have the same proof as 
Claims \ref{c-p17-1}, \ref{c-p17-2}, \ref{c-p17-3} in the proof of Theorem \ref{t-p19}. 

\begin{claim1} \label{c-p23-1}
For any variable $x_i\in X$, $i=1,\ldots,|X|$, both $V_{x_i}$ and $U_{x_i}$ are monochromatic. Furthermore, $S_1$, $S_2$ and $S_3$ are also monochromatic.
\end{claim1}

\begin{claim1}\label{c-p23-2}
It holds that $S_1$ and $S_2$ have different colours.
\end{claim1}

\begin{claim1}\label{c-p23-3}
For every variable $x_i\in X$, $i\in\{1,\ldots,|X|\}$, the cliques $U_{x_i}$ and $V_{x_i}$ have different colours.
\end{claim1}

\begin{claim1}\label{c-p23-4}
For every clause $c_j \in C$, $j\in\{1,\ldots,|C|\}$, exactly two neighbours of $v^1_{c_j}$ (resp.~$v^2_{c_j}$) have the same colour as $v^1_{c_j}$ (resp. $v^2_{c_j}$).
\end{claim1}

\begin{claimproof}
We may assume, without loss of generality, that $v^1_{c_j}$ (resp. $v^2_{c_j}$) is coloured blue.
By symmetry, it is enough to prove the claim for $v^1_{c_j}$. Let $ c_j = (x_{i_1} \vee x_{i_2} \vee x_{i_3})$. Let $v^1_{x_{i_1}},v^2_{x_{i_2}},v^3_{x_{i_3}}$ be the three neighbours of $v^1_{c_j}$ outside of $S_3$ and let $u^1_{x_{i_1}}, u^2_{x_{i_2}}$ and $u^3_{x_{i_3}}$ be the neighbours of the auxiliary vertices $a^k_{c_j}$, $k=1,\ldots,6$. By definition of a perfect-extendable red-blue-colouring, $v^1_{c_j}$ has at least two neighbours outside of $S_3$ that are coloured blue. Suppose for a contradiction that the vertices $v^1_{x_{i_1}},v^2_{x_{i_2}},v^3_{x_{i_3}}$ are all coloured blue. Then, it follows from Claim~\ref{c-p23-1} that $V_{x_{i_1}}, V_{x_{i_2}}, V_{x_{i_3}}$  must be coloured blue. Notice that since $v^1_{c_j}$ is coloured blue, Claim~\ref{c-p23-1} also implies that all vertices in $S_3$,  in particular $u^1_{c_j}$ and $u^2_{c_j}$ are coloured blue.  Let $A_1$ (resp. $A_2$) be the set of auxiliary vertices which are a neighbour of $u^1_{c_j}$ (resp. $u^2_{c_j}$). Then, at least two vertices in $A_1$ and two vertices in $A_2$ are coloured blue. Since $u^1_{x_{i_1}}, u^2_{x_{i_2}}$ and $u^3_{x_{i_3}}$ have each one neighbour in $A_1$ and one neighbour in $A_2$, it follows that one of them has two blue neighbours in $A_1\cup A_2$, and is thus coloured blue. We may assume, without loss of generality, that this vertex is $u^1_{x_{i_1}}$. Using Claim~\ref{c-p23-1} again, we get that $U_{x_{i_1}}$ is coloured blue, a contradiction to Claim~\ref{c-p23-3}.

Since $v^1_{c_j}$ and $v^2_{c_j}$ are both in $S_3$, using Claim~\ref{c-p23-1} we get that they are both blue. For each neighbour of $v^1_{c_j}$ there is a neighbour of $v^2_{c_j}$ in the same clique and vice versa. Hence, they have the same number of blue neighbours.
\end{claimproof}

\noindent
We continue as follows. By Claim~\ref{c-p23-1}, we may assume without loss of generality that $S_3$ is coloured blue. We set every variable $x_i\in X$, for which $V_{x_i}$ has been coloured red, to true. We set all other variables to false. By Claim~\ref{c-p23-4}, we know that for each clause $c_j \in C$, $j\in\{1,\ldots,|C|\}$, there is exactly one red neighbour of $v^1_{c_j}$ (resp. $v^2_{c_j}$). Hence, in every clause, exactly one literal is set to true. Since by Claim~\ref{c-p23-1}, every~$V_{x_i}$, $i\in \{1,\ldots,|X|\}$, is monochromatic, it follows that no variable is both true and false. Thus, $\mathcal{I}$ admits a truth assignment such that each clause contains exactly one true literal.

\medskip
\noindent
Conversely, assume now that $\mathcal{I}$ admits a truth assignment such that each clause contains exactly one true literal. For every variable $x_i\in X$, $i\in\{1,\ldots,|X|\}$, that is set to true, we colour $V_{x_i}$ red and $U_{x_i}$ blue; for every other variable $x_i\in X$, we colour $V_{x_i}$ blue and $U_{x_i}$ red. It follows that every vertex $v^k_{c_j}$, $c_j\in C$ and $j\in\{1,\ldots,|C|\}, k=1,2$, has exactly one red and two blue neighbours outside of $S_3$, since the truth assignment is such that each clause contains exactly one true literal. Thus, we colour $S_3$ blue.  If we consider a clause $c_j = (x_{i_1} \vee x_{i_2} \vee x_{i_3})$, we know that exactly one of $U_{x_{i_1}}, U_{x_{i_2}}, U_{x_{i_3}}$ is blue and the other ones are red.  Assume, without loss of generality, that $U_{x_{i_1}}$ is blue and consider $a^1_{c_{j}}, a^4_{c_{j}}$, the neighbours of  $u^1_{x_{i_1}}\in U_{x_{i_1}}$ and $a^7_{c_{j}}, a^{10}_{c_{j}}$, the neighbours of  $u^4_{x_{i_1}}\in U_{x_{i_1}}$.  We then colour $a^1_{c_{j}}, a^4_{c_{j}},a^7_{c_{j}}, a^{10}_{c_{j}}$ blue, since they have two blue neighbours ($\{u^1_{x_{i_1}},u^1_{c_j}\}$, $\{u^1_{x_{i_1}},u^2_{c_j}\}$, $\{u^4_{x_{i_1}},u^3_{c_j}\}$ resp. $\{u^4_{x_{i_1}},u^4_{c_j}\}$). 
To obtain a perfect-extendable red-blue colouring of $G_{\mathcal{I}}$, we have to colour one of the vertices $a^2_{c_{j}}, a^5_{c_{j}}$ blue and the other one red, and similarly, one of the vertices $a^3_{c_{j}}, a^6_{c_{j}}$ blue and the other one red. Since $u^1_{c_j},u^2_{c_j}$ are both coloured blue and each of them can have at most one red neighbour, we colour $a^2_{c_{j}}, a^6_{c_{j}}$ blue and $a^3_{c_{j}}, a^5_{c_{j}}$ red. With the same arguments we colour $a^8_{c_{j}}, a^{12}_{c_{j}}$ blue and $a^9_{c_{j}}, a^{11}_{c_{j}}$ red.  Finally, the only vertices that remain uncoloured are the vertices in $S_1$ and $S_2$. Here, the only restriction is that $S_1$ and $S_2$ are coloured differently, hence we colour for instance $S_1$ blue and $S_2$ red.  This clearly gives us a valid red-blue-colouring of $G_{\mathcal{I}}$.

\begin{figure}[ht]
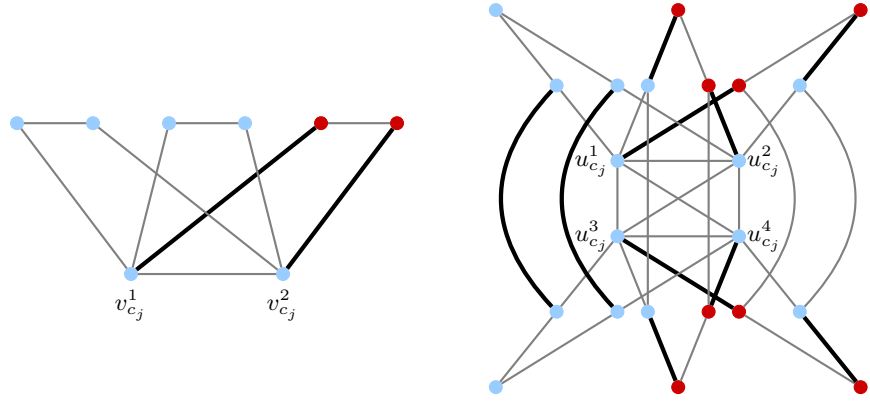

\centering
\include{p23-perfect}
\caption{\label{fig-p23-perfect}A perfect-extendable red-blue-colouring of the clause gadget. Notice that vertices $v_{c_j}^1,v_{c_j}^2,u^1_{c_j},\ldots,u^4_{c_j}$ belong to the clique $S_3$, but for the sake of readability, we omitted the edges between vertices $v_{c_j}^1,v_{c_j}^2$ and vertices $u^1_{c_j},\ldots,u^4_{c_j}$ in this figure.}
\end{figure}

It still remains to verify that this valid red-blue-colouring is perfect-extendable. First, consider the cliques $S_1$ and $S_2$. We know that each vertex $s_j^i$, for $j = 1,2, i \in \set{1,\dots, |X|}$ has two neighbours $v_{x_i}^s$ and $u_{x_i}^s$ outside of $S_1\cup S_2$. Since the red-blue-colouring is valid, it follows that exactly one of $s_1^i$, $s_2^i$ and exactly one of $v_{x_i}^s$ and $u_{x_i}^s$ is blue. So every vertex in $S_1$ and $S_2$ has a neighbour of the opposite colour and hence, can be matched with it. The same holds for the vertices $v_{x_i}^s,u_{x_i}^s$, for all $x_i\in X$. 

Next, consider a clause $c_j\in C$, such that $c_j=x_{i_1}\vee x_{i_2}\vee x_{i_3}$. It follows from the construction of our valid red-blue-colouring that we may assume, without loss of generality, that $V_{x_{i_1}}, U_{x_{i_2}}, U_{x_{i_3}}$ are coloured red and $V_{x_{i_2}}, V_{x_{i_3}}, U_{x_{i_1}}$ are coloured blue. Since every clause vertex is coloured blue and has exactly one red neighbour, it follows that all clause vertices in the clause gadget of $c_j$ can be matched and the same holds for the vertices in $V_{x_{i_1}}$. The auxiliary vertices which are adjacent to vertices in $U_{x_{i_2}}\cup U_{x_{i_3}}$ have exactly one neighbour of the opposite colour and thus, they can be matched with either clause vertices or vertices in $U_{x_{i_2}}\cup U_{x_{i_3}}$ (see Figure \ref{fig-p23-perfect}). The auxiliary vertices $a^1_{c_{j}}, a^4_{c_{j}},a^7_{c_{j}}$ and $a^{10}_{c_{j}}$, which are neighbours of $U_{x_{i_1}}$, are all coloured blue. In this case, we consider the edges $a^1_{c_{j}} a^7_{c_{j}}$ and $a^4_{c_{j}} a^{10}_{c_{j}}$ as matching edges. 

The only vertices that remain unmatched are vertices in variable gadgets of variables $x_i \in X, i \in \set{1,\dots, |X|}$, which are set to false. Since all vertices $v_{x_i}^s$ and $v_{u_i}^s$ are already matched, for all $x_i \in X, i \in \set{1,\dots, |X|}$, there remains an even number of unmatched vertices in each clique of these variable gadgets, all coloured with the same colour. Hence, we can easily find matching edges inside these cliques. We conclude that our valid red-blue-colouring is indeed perfect-extendable.

\medskip
\noindent
To complete the proof, it remains to show that $G_{\mathcal{I}}$ is indeed $(4P_7,P_{23})$-free. To show that it is $P_{23}$-free, we follow the same arguments as in the proof of Theorem \ref{t-p23}, when we showed that the corresponding graph was $P_{19}$-free. Let $P$ be a longest induced path in $G_{\mathcal{I}}$.  $P$ can contain at most two vertices from each clique, since otherwise it would not be induced. Also, if $P$ contains two vertices from a same clique, then these vertices are necessarily consecutive in $P$. Let $W_1,W_2 \in \set{U_{x_i},V_{x_i} |\ x_i \in X}$. Every path from a vertex in $W_1$ to a vertex in $W_2$ contains at least one vertex from one of the cliques $S_1,S_2,S_3$. Hence, $P$ can intersect at most 4 cliques belonging to some variable gadgets; further, if it does intersect four cliques belonging to some variable gadgets, then it intersects each of the cliques $S_1,S_2,S_3$ as well. Notice that after respectively before intersecting $S_3$, $P$ may contain at most two auxiliary vertices (see Figure~\ref{fig-p23-path}). Also, the first two, respectively the last two, vertices in $P$ may correspond to auxiliary vertices. So $P$ contains at most 22 vertices, and thus, $G_{\mathcal{I}}$ is indeed $P_{23}$-free.

\begin{figure}[ht]
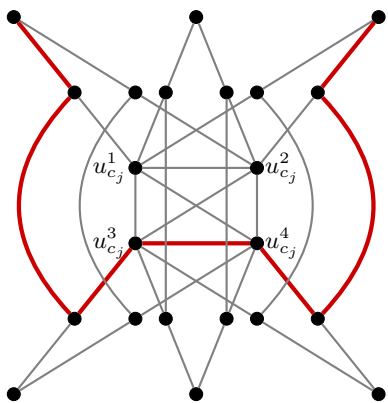

\centering
\include{p23-path}
\caption{\label{fig-p23-path}A path intersecting a clause gadget and containing four auxiliary vertices.}
\end{figure}

Consider now the graph $G' = G [ V(G_\mathcal{I})\setminus (V(S_1) \cup V(S_2) \cup V(S_3))]$. This graph consists of two types of 
connected components: (i) every $V_{x_i}$, for $x_i \in X$, corresponds to a clique of size~$7$; (ii)~every $U_{x_i}$, for $x_i\in X$, together with some auxiliary vertices corresponds to a connected component as shown in Figure~\ref{fig-p23-comps}. Clearly, any induced path contains at most two vertices of a same clique. Furthermore, it is easy to see that any induced path contains at most six vertices from a connected component containing a clique $U_{x_i}$ (see Figure~\ref{fig-p23-comps}). 
Hence, every induced path with length at least~$6$ has to contain a vertex in one of the cliques $S_1,S_2,S_3$. It immediately follows that $G_\mathcal{I}$ is $4P_7$-free.
\end{proof}

\begin{figure}[ht]
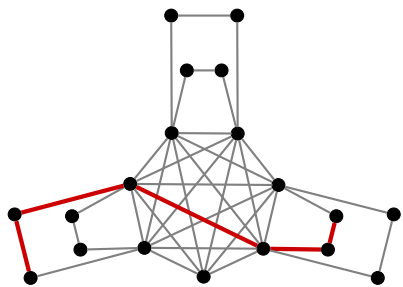

\centering
\include{p23-comps}
\caption{\label{fig-p23-comps}A connected component of type (ii) obtained from the removal of $S_1$, $S_2$ and $S_3$. A longest path in this connected component is shown in red.}
\end{figure}

\section{Our Polynomial Results}\label{s-main3}

We start with two lemmas, whose proofs are similar to the proofs for valid but not necessarily perfect red-blue colourings;  see, for example,~\cite{Fe23} or~\cite{LPR} for explicit proofs.

\begin{lemma}\label{l-dom}
For every integer $g$, it is possible to find in $O(2^gn^{g+2})$ time a perfect red-blue colouring (if it exists) of a graph with $n$ vertices and with domination number~$g$.
\end{lemma}
\begin{proof}
Let $g\geq 1$ be an integer, and let $G$ be a graph with domination number at most $g$. Hence, $G$ has a dominating set $D$ of size at most $g$. We consider all options of colouring the vertices of $D$ red or blue; note that this number is $2^{|D|}\leq 2^g$. For every red vertex of $D$ with no blue neighbour, we consider all $O(n)$ options of colouring exactly one of its neighbours blue (and thus, all of its other neighbours will be coloured red).
Similarly, for every blue vertex of $D$ with no red neighbour, we consider all $O(n)$ options of colouring exactly one of its neighbours red (and thus, all of its other neighbours will be coloured blue). Finally, for every red vertex in $D$ with already one blue neighbour in $D$, we colour all its yet uncoloured neighbours red. 
Similarly, for every blue vertex in $D$ with already one red neighbour in $D$, we colour all its yet uncoloured neighbours blue.

As $D$ is a dominating set, the above means that we guessed a red-blue colouring of the whole graph $G$. We can check in $O(n^2)$ time if a red-blue colouring of a graph with $n$ vertices is perfect. Moreover, the total number of red-blue colourings that we must consider is $O(2^gn^g)$.
\end{proof}

\begin{lemma}\label{l-dom2}
Let $D$ be a dominating set of a connected graph $G$. It is possible to check in polynomial time if $G$ has a perfect red-blue colouring in which $D$ is monochromatic.
\end{lemma}

\begin{proof}
Consider a perfect red-blue colouring~$c$ of $G$, in which $D$ is monochromatic, say every vertex of $D$ is coloured red. Let $F_1,\ldots,F_r$, for some integer $r\geq 1$, be the connected components of $G-D$.

We claim that every $F_i$ is monochromatic. This follows from the same argument as the one for valid red-blue colourings that are not necessarily perfect (see~\cite{LPR}) and we give it for completeness. For a contradiction, assume that, say, $F_1$ is not monochromatic. This means that $F_1$ contains an edge $uv$ where $u$ is coloured red and $v$ is coloured blue. As $D$ is a dominating set of $G$, we have that $v$ also has a neighbour in $D$, which is coloured red. Hence, $c$ is not valid and thus not perfect either, a contradiction. 

Now suppose that some vertex~$u$ in $F_i$ is coloured red. By the above claim, every vertex of $F_i$ is coloured red. The neighbours of $u$ outside $F_i$ are all in $D$ and thus, are coloured red as well. Hence, $u$ has no blue neighbour, meaning that $c$ is not perfect, a contradiction. We conclude that every vertex in $G-D$ must be coloured blue. 

Hence, in order to check if $G$ has a perfect red-blue colouring in which $D$ is monochromatic, we can do as follows. Colour every vertex in $D$ red and colour every vertex in $G-D$ blue. Then, check in polynomial time if the resulting red-blue colouring is perfect.
\end{proof}

\noindent
To handle ``partial'' perfect red-blue colourings, we introduce the following terminology.
Let $G=(V,E)$ be a connected graph and $S,T,X,Y\subseteq V$ be four non-empty sets with $S\subseteq X$, $T\subseteq Y$ and $X\cap Y=\emptyset$. A  {\it red-blue $(S,T,X,Y)$-colouring} of $G$ is a red-blue colouring where
\begin{itemize}
\item every vertex of $X$ is coloured red and every vertex of $Y$ is coloured blue;
\item the blue neighbour of every vertex in $S$ belongs to $T$ and vice versa; and
\item the blue neighbour of every vertex in $X\setminus S$ and the red neighbour of every vertex of $Y$ belong to $V\setminus (X\cup Y)$.\\[-10pt]
\end{itemize}
 For a connected graph $G=(V,E)$, let $S'$ and $T'$ be two disjoint subsets of $V$, such that (i)~every vertex of $S'$ is adjacent to at most one vertex of $T'$, and vice versa, and (ii) at least one vertex in $S'$ is adjacent to a vertex in $T'$.  
Let $S''$ consist of all vertices of $S'$ with a (unique) neighbour in $T'$, and let $T''$ consist of all vertices of $T'$ with a (unique) neighbour in $S'$ (so, every vertex in $S''$ has a unique neighbour in $T''$, and vice versa). We call $(S'',T'')$ the {\it core} of {\it starting pair} $(S',T')$; note that $|S''|=|T''|\geq 1$.

We colour every vertex in $S'$ red and every vertex in $T'$ blue. Propagation rules will try to extend $S'$ and $T'$ by finding new vertices whose colour must always be either red or blue. We place new red vertices in a set $X$ and new blue vertices in a set $Y$. If a red and blue vertex are matched to each other, then we add the red one to a set $S\subseteq X$ and the blue one to a set $T\subseteq Y$.
Initially, $S:=S''$, $T:=T''$, $X:=S'$ and $Y:=T'$, and we let $Z:=V\setminus (X\cup Y)$.

We now present seven propagation rules for finding perfect red-blue $(S,T,X,Y)$-colourings.
Rules R1 and R2 hold for finding red-blue colourings in general and correspond to the five rules from~\cite{LL19}. Rules R3-R7 are for finding perfect red-blue colourings; some of them are in a slightly different form in~\cite{LT21}.  
\begin{enumerate}
\item [\bf R1.] Return {\tt no} (i.e., $G$ has no red-blue $(S,T,X,Y)$-colouring) if a vertex $v\in Z$ is
\begin{itemize}
\item [(i)]   adjacent to a vertex in $S$ and to a vertex in $T$, or
\item [(ii)]  adjacent to a vertex in $S$ and to two vertices in $Y\setminus T$, or
\item [(iii)]   adjacent to a vertex in $T$ and to two vertices in $X\setminus S$, or
\item  [(iv)] adjacent to two vertices in $X\setminus S$ and to two vertices in $Y\setminus T$.\vspace*{-3mm}
\end{itemize}
\end{enumerate}

\begin{enumerate}
\item [{\bf R2.}]  Let $v\in Z$.
\begin{itemize}
\item [(i)] If $v$ is adjacent to a vertex in $S$ or to two vertices of $X\setminus S$, then move $v$ from $Z$ to $X$. If moreover $v$ is also adjacent to a vertex $w$ in $Y$, then add $v$ to $S$ and $w$ to $T$.
\item [(ii)] If $v$ is adjacent to a vertex in $T$ or to two vertices of $Y\setminus T$, then move $v$ from $Z$ to $Y$. If moreover $v$ is also adjacent to a vertex $w$ in $X$, then add $v$ to $T$ and $w$~to~$S$. \vspace*{-3mm}
\end{itemize}
\end{enumerate}

\begin{enumerate}
\item [{\bf R3.}] Let $v\in (X\cup Y)\setminus (S\cup T)$.
\begin{itemize}
\item [(i)] If $v\in X\setminus S$ and $v$ is adjacent to a vertex~$w$ in $Y$, then add $v$ to $S$ and $w$ to~$T$.
\item [(ii)] If $v\in Y\setminus T$ and $v$ is adjacent to a vertex~$w$ in $X$, then add $v$ to $T$ and $w$ to~$S$. \vspace*{-3mm}
\end{itemize}
\end{enumerate}

\begin{enumerate}
\item [{\bf R4.}] Return {\tt no} if
\begin{itemize}
\item [(i)] a vertex $x\in X$ has no neighbours outside $X$ or is adjacent to two vertices of $Y$, or
\item [(ii)] a vertex $y\in Y$ has no neighbours outside $Y$, or is adjacent to two vertices of $X$. \vspace*{-3mm}
\end{itemize}
\end{enumerate}

\begin{enumerate}
\item [{\bf R5.}]  Let $v\in Z$. Let $w \in Z$ be a vertex with $N_G(w)=\{v\}$.
\begin{itemize}
\item [(i)] If $v$ is adjacent to a vertex in $X$ and to a vertex in $Y$, then return {\tt no}.
\item [(ii)] If $v$ is adjacent to a vertex in $X$ but not to a vertex in $Y$, then put $v$ in $X$ and $w$ in~$Y$, and also add $v$ to $S$ and $w$ to $T$.
\item [(iii)] If $v$ is adjacent to a vertex in $Y$ but not to a vertex in $X$, then put $v$ in $Y$ and $w$ in~$X$, and also add $v$ to $T$ and $w$ to $S$. \vspace*{-3mm}
\end{itemize}
\end{enumerate}

\begin{enumerate}
\item [{\bf R6.}]  Let $v\in Z$ be in a connected component $F$ of $G[Z]$ such that $F$ is isomorphic to $C_4$.
\begin{itemize}
\item [(i)] If $v$ is adjacent to a vertex in $X$ but not to a vertex in $Y$, and $F$ contains a vertex not adjacent to a vertex in $X$,
then move $v$ from $Z$ to~$X$.
\item [(ii)] If $v$ is adjacent to a vertex in $Y$ but not to a vertex in $X$, and $F$ contains a vertex not adjacent to a vertex in $Y$,
then move $v$ from $Z$ to~$Y$. \vspace*{-3mm}
\end{itemize}
\end{enumerate}

\begin{enumerate}
\item [{\bf R7.}]  Let $v\in Z$ be in a connected component $F$ of $G[Z]$ such that $\{v\}$ dominates $F$. Let $F-v$ have a vertex~$w$ with only one neighbour $w'$ in $X\cup Y$.
\begin{itemize}
\item [(i)] If $w'\in X$, then put $v$ in $Y$.
\item [(ii)] If $w'\in Y$, then put $v$ in $X$. 
\end{itemize}
\end{enumerate}

\noindent
A propagation rule is {\it safe} if the input graph has a perfect red-blue $(S,T,X,Y)$-colouring before the application of the rule if and only if it has so after the application of the rule. 

\begin{lemma}\label{l-safe}
Rules R1--R7 are safe.
\end{lemma}
\begin{proof}
Let $G$ be a connected graph with a perfect red-blue $(S,T,X,Y)$-colouring. First recall that, by definition, vertices in $X$ will be coloured red by every red-blue $(S,T,X,Y)$-colouring, whilst vertices of $Y$ will be coloured blue, and moreover that every (red) vertex in $S$ has exactly one (blue) neighbour in $T$, and vice versa. The colour of the vertices in $Z$ still has to be decided.

Rule R1-(i) is safe. A vertex adjacent to both a red vertex that already has a blue neighbour and to a blue vertex that already has a red neighbour can be coloured neither red nor blue. Rule R1-(ii) is safe, as a vertex that is adjacent to a red vertex that already has a blue neighbour must be coloured red, so it cannot also be adjacent to two blue vertices. For the same reason, R1-(iii) is safe. Finally, R1-(iv) is safe, as a vertex that is adjacent to two red vertices must be coloured red, so it cannot also be adjacent to two blue vertices. 

Rule R2-(i) is safe. Any vertex adjacent to a red vertex that already has a blue neighbour or to two red vertices must be coloured red. If such a vertex is adjacent to a vertex coloured blue already, it will have its blue neighbour and thus must be added to $S$, whilst its blue neighbour must be added to $T$. For the same reason, R2-(ii) is safe as well. 

Rule R3-(i) is safe. Every red vertex must have a (unique) blue neighbour, and vice versa. For the same reason, R3-(ii) is safe. 

Rule R4-(i) is safe. In the first case, $x$ will only have red neighbours in $G$ (as $x$ is coloured red, $x$ needs a blue neighbour as well). In the second case, $x$ will have two blue neighbours, while $x$ is coloured red. This is not possible. For the same reason R4-(ii) is safe as well.

Rule R5-(i) is safe. As $v$ will be adjacent to a blue and red neighbour, $w$ cannot be its matching neighbour. As $w$ has degree~$1$, we find that $w$ cannot be matched. Rule R5-(ii) is safe as well. For a contradiction, suppose that we would put $v$ in $Y$. Then the matched neighbour of $v$ is the neighbour of $v$ that belongs to $x$. Hence, again we find that $w$ does not have a matched neighbour. So we must put $v$ in $X$, and then $v$ and $w$ will be matched to each other. For the same reason, R5-(iii) is safe.

Rule R6-(i) is safe. Suppose that we put $v$ in $Y$, that is, $v$ will be coloured blue. Consequently, $v$ has its matching neighbour in $X$.  This means that the two neighbours of $v$ in $F$ will be coloured blue. As $F$ is a cycle on four vertices, the fourth vertex will get colour blue as well. By assumption, $F$ contains a vertex that is not adjacent to a vertex in $X$. This vertex is coloured blue, but will not have a red neighbour. We conclude that $v$ must be put in $X$.
For the same reason, R6-(ii) is safe as well. 

Rule R7-(i) is safe. Suppose that we put $v$ in $X$, so $v$ will be coloured red, just like $w'$. Hence, $w$ is adjacent to two red vertices, and must be coloured red as well. As $w'$ is the only neighbour of $w$ in $X\cup Y$, we find that every other neighbour of $w$ is in $F$. As $\{v\}$ dominates $F$, this means that such a neighbour of $w$ is adjacent not only to $w$ but also to $v$, and hence must be coloured red. This means that $w$ will not have any blue neighbour. We conclude that $v$ must be put in $Y$. For the same reason, R7-(ii) is safe as well. 
\end{proof}

\noindent
Assume that exhaustively applying rules R1--R7 on a starting pair $(S',T')$ did not lead to a no-answer but to a $4$-tuple $(S,T,X,Y)$. Then we call $(S,T,X,Y)$ an {\it intermediate} $4$-tuple. The first part of the next lemma follows from Lemma~\ref{l-safe}. The second part is straightforward.

\begin{lemma}\label{l-iff}
Let $G$ be a graph with a starting pair $(S',T')$ with core $(S'',T'')$ and a resulting intermediate $4$-tuple $(S,T,X,Y)$. Then $G$ has a perfect red-blue $(S'',T'',S',T')$-colouring if and only if $G$ has a perfect red-blue $(S,T,X,Y)$-colouring. Moreover, $(S,T,X,Y)$ can be obtained in polynomial time.
\end{lemma}
\begin{proof}
The first part of the lemma follows from Lemma~\ref{l-safe} and our initialisation. To prove the running time statement, we first note that each application of R1--R7 takes polynomial time. For each rule we can also check in polynomial time if it can be applied. Moreover, after each application of a rule we either find a no-answer or reduce the size of at least one of the sets $X$, $Y$, $Z$. Hence, we obtained $(S,T,X,Y)$ in polynomial time.
\end{proof}

\noindent
We now describe the structure of a graph with a $4$-tuple $(S,T,X,Y)$.

\begin{lemma}\label{l-lll1b}
Let $G$ be a graph with an intermediate $4$-tuple $(S,T,X,Y)$. Then:
\begin{itemize}
\item [(i)] every vertex in $S$ has exactly one neighbour in $Y$, which belongs to $T$;
\item [(ii)] every vertex in $T$ has exactly one neighbour in $X$, which belongs to $S$;
\item [(iii)] every vertex in $X\setminus S$ has no neighbour in $Y$;
\item [(iv)] every vertex in $Y\setminus T$ has no neighbour in $X$;
\item [(v)] every vertex in $V\setminus (X\cup Y)$ has no neighbour in $S\cup T$, at most one neighbour in $X\setminus S$ and at most one neighbour in $Y\setminus T$.
\end{itemize}
\end{lemma}
\begin{proof}
Let $Z=V\setminus (X\cup Y)$. We prove each of the statement below one by one.

\medskip
\noindent
{\it Proof of (i).} For a contradiction, first assume that some vertex~$u$ in $S$ has no neighbour in $T$. Then $u$ has no neighbour in $Y\setminus T$ either, else we would have applied R3. However, now we would have applied R4 (and returned a no-answer). Hence, every vertex in $S$ has a neighbour in $T\subseteq Y$. If a vertex in $S$ has more than one neighbour in $Y$, then we would have applied R4 as well. 

\medskip
\noindent
{\it Proof of (ii).} Statement (ii) follows by symmetry: we can use the same arguments as in the proof of (i). 

\medskip
\noindent
{\it Proof of (iii).} Let $u\in X\setminus S$. If $u$ has a neighbour in $Y$, then we would have applied R3. Hence, $u$ has no neighbours in $Y$. 

\medskip
\noindent
{\it Proof of (iv).} Statement (iv) follows by symmetry: we can use the same arguments as in the proof of (iii).

\medskip
\noindent
{\it Proof of (v).}  Let $u\in Z$. If $u$ is adjacent to a vertex in $S\cup T$, then we would have applied R2. If $u$ is adjacent to two vertices in $X\setminus S$ or to two vertices in $Y\setminus T$, then we would also have applied~R2. 
\end{proof}

\noindent
Let $(S,T,X,Y)$ be an intermediate $4$-tuple of a graph $G$. Let $Z=V\setminus (X\cup Y)$. A red-blue $(S,T,X,Y)$-colouring of $G$ is {\it monochromatic} if all connected components of $G[Z]$ are monochromatic.  Rules R8-R11 preserve this property; some of them were also used in~\cite{LL19,LT21}. 

\begin{enumerate}
\item [{\bf R8.}]
Let $v\in Z$. If $v$ is not adjacent to any vertex of $X\cup Y$, then return {\tt no}. \vspace*{-3mm}
\end{enumerate}

\begin{enumerate}
\item [{\bf R9.}]  Let $v\in Z$ be a vertex in a connected component $F$ of $G[Z]$ such that $v$ has only one neighbour~$w$ in~$X\cup Y$.
\begin{itemize}
\item [(i)] If $w\in X$, then put every vertex of $F$ in $Y$ and also add every vertex of $F$ to $T$ and every neighbour of every vertex of $F$ in $X$ to $S$.
\item [(ii)] If $w\in Y$, then put every vertex of $F$ in $X$ and also add every vertex of $F$ to $S$ and every neighbour of every vertex of $F$ in $Y$ to $T$. \vspace*{-3mm}
\end{itemize}
\end{enumerate}

\begin{enumerate}
\item [{\bf R10.}]  Let $v\in (X\cup Y)\setminus (S\cup T)$ and $F$ be a connected component of $G[Z]$ such that $v$ has two neighbours in $F$.
\begin{itemize}
\item [(i)] If $v\in X\setminus S$, then put every vertex of $F$ in $X$, and also add every vertex of $F$ to $S$ and every vertex of every neighbour of $F$ in $Y\setminus T$ to $T$.
\item [(ii)] If $v\in Y\setminus T$, then put every vertex of $F$ in $Y$, and also add every vertex of $F$ to $T$ and every vertex of every neighbour of $F$ in $X\setminus S$ to $S$. \vspace*{-3mm}
\end{itemize}
\end{enumerate}

\begin{enumerate}
\item [{\bf R11.}]   Let $v\in (X\cup Y)\setminus (S\cup T)$ and $F$ be a connected component of $G[Z]$ such that $v$ has one neighbour in $F$ that is the only neighbour of $v$ in $Z$.
\begin{itemize}
\item [(i)] If $v\in X\setminus S$ and $v$ is not adjacent to $Y$, then put every vertex of $F$ in $Y$, and also add every vertex of $F$ to $T$ and every vertex of every neighbour of $F$ in $X\setminus S$ to $S$.
\item [(ii)] If $v\in Y\setminus T$ and $v$ is not adjacent to $X$, then put every vertex of $F$ in $X$, and also add every vertex of $F$ to $S$ and every vertex of every neighbour of $F$ in $Y\setminus T$ to $T$.
\end{itemize}
\end{enumerate}

\noindent
A propagation rule is {\it mono-safe} if the input graph has a (monochromatic) perfect red-blue $(S,T,X,Y)$-colouring before the application of the rule if and only if it has so after the application of the rule. The following lemma is not difficult to prove.

\begin{lemma}\label{l-safe2}
Rules R8--R11 are mono-safe.
\end{lemma}
\begin{proof}
Let $G$ be a connected graph with a monochromatic perfect red-blue $(S,T,X,Y)$-colouring. First recall that, by definition, vertices in $X$ will be coloured red by every red-blue $(S,T,X,Y)$-colouring, whilst vertices of $Y$ will be coloured blue, and moreover that every (red) vertex in $S$ has exactly one (blue) neighbour in $T$, and vice versa. The colour of the vertices in $Z$ still has to be decided.

Rule R8 is mono-safe. Let $F$ be the connected component of $G[Z]$ that contains $v$. Then all neighbours of $v$ belong to $F$, which must be monochromatic. Thus the neighbours of $v$ are coloured either all red or all blue. Hence, $v$ will not have the required neighbour with a different colour than itself. 

Rule R9-(i) is mono-safe. As all vertices in $F$ will be coloured with the same colour, this means that $w$ must receive a different colour than $v$. Hence, as $w$ is coloured red, we find that $v$, and thus all other vertices of $F$, must be coloured blue. As every vertex of $F$ only has neighbours in $F$ and in $X\cup Y$, we find that all neighbours of every vertex in $F$ are coloured. Hence, we can identify the unique red neighbours of the vertices of $F$, which in turn will be the unique blue neighbours of these vertices. For the same reason Rule R9-(ii) is mono-safe as well. 

Rule R10-(i) is mono-safe. All vertices in $F$ will be coloured with the same colour and at least two of them are adjacent to $v$. Hence, the vertices in $F$ must all get the same colour as the colour of $v$, which is red. Just as in the previous rule, we can now identify the unique blue neighbours of the vertices of $F$, which in turn will be the unique red neighbours of these vertices. For the same reason Rule R10-(ii) is safe as well. 

Rule R11-(i) is mono-safe. Let $w$ be the neighbour of $v$ in $F$. Then all other neighbours of $v$ belong to $X$; else $v$ would have a neighbour in $Y$ and we would have applied R3. As all vertices in $X$ are coloured red, this means that $w$ must be coloured blue. Hence, as every vertex of $F$ will be coloured the same, every vertex of $F$ will be coloured blue. Just as in the previous two rules, we can now identify the unique red neighbours of the vertices of $F$, which in turn will be the unique blue neighbours of these vertices. For the same reason Rule R11-(ii) is mono-safe.
\end{proof}

\noindent
Suppose exhaustively applying rules R1--R11 on an intermediate $4$-tuple $(S,T,X,Y)$ did not lead to a no-answer but to a $4$-tuple $(S^*,T^*,X^*,Y^*)$. We call $(S^*,T^*,X^*,Y^*)$ the {\it final} $4$-tuple.
The first part of Lemma~\ref{l-iff2} follows from Lemma~\ref{l-safe2}. The second part is straightforward.

\begin{lemma}\label{l-iff2}
Let $G$ be a graph with an intermediate $4$-tuple $(S,T,X,Y)$ and a resulting final $4$-tuple  $(S^*,T^*,X^*,Y^*)$.  Then $G$ has a monochromatic perfect red-blue $(S,T,X,Y)$-colouring if and only if $G$ has a monochromatic perfect red-blue $(S^*,T^*,X^*,Y^*)$-colouring. Moreover, $(S^*,T^*,X^*,Y^*)$ can be obtained in polynomial time.
\end{lemma}
\begin{proof}
The first part of the lemma follows from Lemma~\ref{l-safe2} and the fact that $(S^*,T^*,X^*,Y^*)$ results from $(S,T,X,Y)$. To prove the running time statement, we first note that each application of R1--R11 takes polynomial time. For each rule we can also check in polynomial time if it can be applied. Moreover, after each application of a rule we either find a no-answer or reduce the size of at least one of the sets $X$, $Y$, $Z$. Hence, we obtained $(S,T,X,Y)$ in polynomial time.
\end{proof}

\begin{figure}[t]
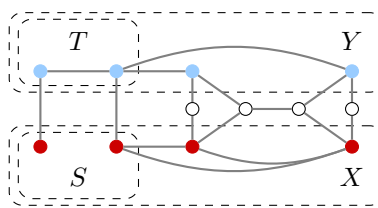

\centering
\include{lemma-rules-morespace}
\caption{A red-blue $(S,T,X,Y)$-colouring of a graph with a final $4$-tuple $(S,T,X,Y)$.}\label{fig-lemma-rules}
\end{figure}

\noindent
We now describe the structure of a graph with a final $4$-tuple $(S,T,X,Y)$; see also Figure~\ref{fig-lemma-rules}.

\begin{lemma}\label{l-ll1b}
Let $G$ be a graph with a final $4$-tuple $(S,T,X,Y)$. The following holds:
\begin{itemize}
\item [(i)] every vertex in $S$ has exactly one neighbour in $Y$, which belongs to $T$;
\item [(ii)] every vertex in $T$ has exactly one neighbour in $X$, which belong to $S$;
\item [(iii)] every vertex in $X\setminus S$ has no neighbour in $Y$, at least two neighbours in $V\setminus (X\cup Y)$ but no two 
neighbours in the same connected component of $G[V\setminus (X\cup Y)]$; 
\item [(iv)] every vertex in $Y\setminus T$ has no neighbour in $X$, at least two neighbours in $V\setminus (X\cup Y)$ but no two 
neighbours in the same connected component of $G[V\setminus (X\cup Y)]$; 
\item [(v)] every vertex of $V\setminus (X\cup Y)$ has no neighbour in $S\cup T$, exactly one neighbour in $X\setminus S$ and exactly one neighbour in $Y\setminus T$.
\end{itemize}
\end{lemma}
\begin{proof}
Let $Z=V\setminus (X\cup Y)$. We prove each of the statement below one by one.

\medskip
\noindent
{\it Proof of (i).} For a contradiction, first assume that some vertex~$u$ in $S$ has no neighbour in $T$. Then $u$ has no neighbour in $Y\setminus T$ either, else we would have applied R3. However, now we would have applied R4 (and returned a no-answer). Hence, every vertex in $S$ has a neighbour in $T\subseteq Y$. If a vertex in $S$ has more than one neighbour in $Y$, then we would have applied R4 as well. 

\medskip
\noindent
{\it Proof of (ii).} Statement (ii) follows by symmetry: we can use the same arguments as in the proof of (i). 

\medskip
\noindent
{\it Proof of (iii).} Let $u\in X\setminus S$. If $u$ has a neighbour in $Y$, then we would have applied R3. Hence, $u$ has no neighbours in $Y$. If $u$ has no neighbours in $V\setminus (X\cup Y)$, then $u$ would only have neighbours in $X$. In that case we would have applied R4 (and returned a no-answer). If $u$ only has one neighbour in $Z$, then we would have applied R11. Hence, $u$ has at least two neighbours in $Z$. If two neighbours of $u$ in $Z$ belong to the same connected component of $Z$, then we would have applied R10. 

\medskip
\noindent
{\it Proof of (iv).} Statement (iv) follows by symmetry: we can use the same arguments as in the proof of (iii).

\medskip
\noindent
{\it Proof of (v).}  Let $u\in Z$.  If $u$ is adjacent to a vertex in $S\cup T$, then we would have applied R2. Hence, we find that $u$ is not adjacent to a vertex in $S\cup T$. If $u$ is adjacent to two vertices in $X\setminus S$ or to two vertices in $Y\setminus T$, then we would also have applied R2. If $u$ has exactly one neighbour in $X\cup Y$, then we would have applied R9.
If $u$ has no neighbour in $X\setminus S$ and no neighbour in $Y\setminus T$, then $u$ has no neighbour in $X\cup Y$, as we already deduced that $u$ has no neighbour in $S\cup T$. However, then we would have applied R8 (and returned a no-answer).
\end{proof}

\noindent
We can now prove a lemma that is the cornerstone for our polynomial-time results.

\begin{lemma}\label{l-mono}
Let $G$ be a graph with a final $4$-tuple $(S,T,X,Y)$. Then it is possible to find in polynomial time a monochromatic perfect red-blue $(S,T,X,Y)$-colouring of $G$ or conclude that such a colouring does not exist.
\end{lemma}

\begin{proof}
Let $Z=V\setminus (X\cup Y)$. Let $E^*\subseteq E$ be the set of edges consisting of all edges with one end-vertex in $(X\cup Y)\setminus (S\cup T)$ and the other end-vertex in $Z$. By Lemma~\ref{l-ll1b}-(v), we find that $|E^*|=2|Z|$. By Lemma~\ref{l-ll1b}-(iii) and~(iv), we find that $|E^*|\geq 2|(X\cup Y)\setminus (S\cup T)|$. Hence, $|Z|\geq |(X\cup Y)\setminus (S\cup T)|$, and $|Z|=|(X\cup Y)\setminus (S\cup T)|$ if and only if each vertex in $(X\cup Y)\setminus (S\cup T)$ has exactly two neighbours in $Z$.

Every vertex $u\in Z$ still needs their matching neighbour $v$. In order for $G$ to have a monochromatic perfect red-blue $(S,T,X,Y)$-colouring, $v$ must be outside $S\cup T$, so $v$ belongs to $X\cup Y$. By Lemma~\ref{l-ll1b}-(v), we find that $v\in (X\cup Y)\setminus (S\cup T)$. As matching neighbours are ``private'', $|Z|\leq |(X\cup Y)\setminus (S\cup T)|$. We conclude that  $|(X\cup Y)\setminus (S\cup T)|=|Z|$. Our algorithm checks this in polynomial time and returns a no-answer if $|(X\cup Y)\setminus (S\cup T)|\neq |Z|$. 

From now on, assume $|(X\cup Y)\setminus (S\cup T)|=|Z|$. Hence, each vertex in $(X\cup Y)\setminus (S\cup T)$ has exactly two neighbours in $Z$. Just like~\cite{LL19}, we now construct an instance $\phi$ of the {\sc $2$-Satisfiability} problem ({\sc $2$-SAT}). 
Our $2$-SAT formula differs from the one in~\cite{LL19} due to the perfectness requirement. 
For each connected component $C$ of $G[Z]$, we do as follows. We define two variables $x_C$ and $y_C$, and we add the clause $(x_C\vee y_C)\wedge (\neg x_C \vee \neg y_C)$ to $\phi$. For each $u\in (X\cup Y)\setminus (S\cup T)$, we do as follows. From the above we known that $u$ has exactly two neighbours $v$ and $w$ in $Z$. Let $C$ be the connected component of $G[Z]$ that contains $v$ and $D$ be the connected component of $G[Z]$ that contains $w$. We add the clause $(x_C\vee x_D)\wedge (y_C\vee y_D)$ to $\phi$. This finishes the construction of $\phi$.

We claim that $G$ has a monochromatic perfect red-blue $(S,T,X,Y)$-colouring if and only if $\phi$ has a satisfying truth assignment.
It is readily seen and well known that {\sc $2$-SAT} is polynomial-time solvable, meaning we are done once we have proven this claim.

First suppose that $G$ has a monochromatic perfect red-blue $(S,T,X,Y)$-colouring~$c$. By definition, the vertices in each connected component~$C$ of $G[Z]$ are coloured alike. We define a truth assignment~$\tau$ as follows.
We let $x_C$ be true if and only if the vertices of $C$ are coloured red. We let $y_C$ be true if and only if the vertices of $C$ are coloured blue. As exactly one of these options holds, the clause $(x_C\vee y_C)\wedge (\neg x_C \vee \neg y_C)$ is satisfied. 

Now consider a clause $(x_C\vee x_D)\wedge (y_C\vee y_D)$ corresponding to a vertex $u\in (X\cup Y)\setminus (S\cup T)$ that has a neighbour in each of the connected components $C$ and $D$ of $G[Z]$. Then, by Lemma~\ref{l-ll1b}-(iii) and~(iv), $C$ and $D$ are different connected components of $G[Z]$. First assume that $u\in X\setminus S$. By Lemma~\ref{l-ll1b}-(iii), we find that $u$ has no neighbour in $Y$ and thus its blue neighbour must either be in $C$ or in $D$. If it is in $C$, then the neighbour of $u$ in $D$ is coloured blue, and vice versa. As $c$ is monochromatic, this means that either all vertices of $C$ are coloured red and all vertices of $D$ are coloured blue, or the other way around. Hence, the clause $(x_C\vee x_D)\wedge (y_C\vee y_D)$ is satisfied. If $u\in Y\setminus T$, we can use exactly the same arguments. We conclude that $\tau$ is a satisfying truth assignment.

Now suppose that $\phi$ has a satisfying truth assignment~$\tau$. For every connected component~$C$ of $G[Z]$, we colour the vertices of $C$ red if $x_C$ is true and we colour the vertices of $C$ blue if $y_C$ is true. As $\tau$ satisfies $(x_C\vee y_C)\wedge (\neg x_C \vee \neg y_C)$, exactly one of $x_C$ or $y_C$ is true. Hence, the colouring of the vertices of $Z$ is well defined. 

We also colour all vertices of $X$ red and all vertices of $Y$ blue. We let $c$ be the resulting colouring. By construction, it is monochromatic. Hence, it remains to show that $c$ is a perfect red-blue $(S,T,X,Y)$-colouring. We will do this below.

First, it follows from the definition of a core $(S'',T'')$ that $S''$ and $T''$ are non-empty. Moreover, before applying the reduction rules, we first do an initiation, from which it follows that $S''\subseteq S$ and $T''\subseteq T$. Hence,  at least one vertex of $G$ is coloured red and at least one vertex of $G$ is coloured blue.

By Lemma~\ref{l-ll1b}-(i), every vertex in $S$ has exactly one neighbour in $Y$. By Lemma~\ref{l-ll1b}-(ii), every vertex in $T$ has exactly one neighbour in $X$. By Lemma~\ref{l-ll1b}-(v), no vertex of $S\cup T$ is adjacent to a vertex of $Z$. Hence, the vertices in $S\cup T$ have exactly one neighbour of opposite colour.

By Lemma~\ref{l-ll1b}-(v), every vertex $z\in Z$ has exactly one neighbour in $X\setminus S$, which is coloured red, and exactly one neighbour in $Y\setminus T$, which is coloured blue; moreover, $z$ is not adjacent to any vertex in $S\cup T$. Let $C$ be the connected component of $G[Z]$ that contains $z$. As $c$ is monochromatic, all vertices of $C$ receive the same colour. Hence, the vertices in $Z$ have each exactly one neighbour of opposite colour. 

Finally, we must verify the vertices in $(X\cup Y)\setminus (S\cup T)$. Let $u\in (X\cup Y)\setminus (S\cup T)$. First assume that $u\in X\setminus S$, so $u$ is coloured red. We recall that $u$ has exactly two neighbours $v$ and $w$ in $Z$. Let $C$ be the connected component of $G[Z]$ that contains $u$, and let $D$ be the connected component of $G[Z]$ that contains $w$. Hence, $\tau$ contains the clause $(x_C\vee x_D)\wedge (y_C\vee y_D)$.
By Lemma~\ref{l-ll1b}-(iii), we find that $C$ and $D$ are two distinct connected components of $G[Z]$. As $\tau$ satisfies $(x_C\vee x_D)\wedge (y_C\vee y_D)$, the vertices of one of $C$, $D$ are coloured red, while the vertices of the other one are coloured blue. 
By Lemma~\ref{l-ll1b}-(iii), we find that $u$ has no (blue) neighbour in $Y$. Hence, $u$ has exactly one blue neighbour. If $u\in Y\setminus T$, we can apply the same arguments. We conclude that also the vertices in $(X\cup Y)\setminus (S\cup T)$ have exactly one neighbour of the opposite colour.

From the above we conclude that $c$ is monochromatic and perfect. 
 \end{proof}

\noindent
We apply Lemma~\ref{l-mono} in the next result. Its proof is similar but more involved than the one for {\sc Matching Cut} on graphs of radius~$2$~\cite{LPR}.

\begin{theorem}\label{t-radius}
 {\sc  Perfect Matching Cut} is polynomial-time solvable for graphs of radius at most~$2$. 
\end{theorem}
\begin{proof}
Let $G$ be a graph of radius $r$ at most $2$.
If $r=1$, then $G$ has a vertex that is adjacent to all other vertices. In this case $G$ has a perfect matching cut if and only if $G$ consists of two vertices with an edge between them. From now on, assume that $r=2$. Then $G$ has a dominating star $H$, say $H$ has centre $u$ and leaves $v_1,\ldots,v_s$ for some $s\geq 1$. By Observation~\ref{o} it suffices to check if $G$ has a perfect red-blue colouring.

We first check if $G$ has a perfect red-blue colouring in which $V(H)$ is monochromatic. By Lemma~\ref{l-dom2} this can be done in polynomial time. Suppose we find no such red-blue colouring. Then we may assume without loss of generality that a perfect red-blue colouring of $G$ (if it exists) colours $u$ red and exactly one of $v_1,\ldots,v_s$ blue.
That is, $G$ has a perfect red-blue colouring if and only if $G$ has a perfect red-blue $(\{u\},\{v_i\},\{u\},\{v_i\})$-colouring for some $i\in \{1,\ldots,s\}$.
We consider all $O(n)$ options of choosing which $v_i$ is coloured blue. 

For each option we do as follows. Let $v_i$ be the vertex of $v_1,\ldots,v_s$ that we coloured blue. We define
the starting pair $(S',T')$ with core $(S',T')$, where $S'=\{u\}$ and $T'=\{v_i\}$. We now apply rules R1--R7 exhaustively. The latter takes polynomial time by Lemma~\ref{l-iff}. If this exhaustive application leads to a no-answer, then by Lemma~\ref{l-iff} we may discard the option. Suppose we obtain an intermediate $4$-tuple $(S,T,X,Y)$. By again applying Lemma~\ref{l-iff}, we find that
$G$ has a perfect red-blue $(\{u\},\{v_i\},\{u\},\{v_i\})$-colouring if and only if $G$ has a perfect red-blue $(S,T,X,Y)$-colouring. By R2-(i) and the fact that
$u\in S'\subseteq S$ we find that $\{v_1,\ldots,v_s\}\setminus \{v_i\}$ belongs to~$X$.

Suppose that $G$ has a perfect red-blue $(S,T,X,Y)$-colouring~$c$ such that $G[V(G)\setminus(X\cup Y)]$ has a connected component~$D$ that is not monochromatic. Then $D$ must contain an edge~$uv$, where $u$ is coloured red and $v$ is coloured blue. 
Note that $v$ cannot be adjacent to $v_i$, as otherwise $v$ would have been in $Y$ by R3 (since $v_i\in T'\subseteq T$). As $H$ is dominating, this means that $v$ must be adjacent to a vertex $w\in V(H)\setminus \{v_i\}=\{u,v_1,\ldots,v_s\}\setminus \{v_i\}$. 
As $u\in S'\subseteq S\subseteq X$ and $\{v_1,\ldots,v_s\}\setminus \{v_i\}\subseteq X$, we find that $w\in X$ by R2-(i) and thus will be coloured red. However, now $v$ being coloured blue is adjacent to two red vertices (namely $u$ and $w$), contradicting the validity of $c$. 

From the above we conclude that every perfect red-blue $(S,T,X,Y)$-colouring of $G$ is monochromatic. 
We now apply rules R1--R11 exhaustively. The latter takes polynomial time by Lemma~\ref{l-iff2}. If this exhaustive application leads to a no-answer, then by Lemma~\ref{l-iff2} we may discard the option. Suppose we obtain a final $4$-tuple $(S^*,T^*,X^*,Y^*)$. By again applying Lemma~\ref{l-iff2}, we find that
$G$ has a monochromatic perfect red-blue $(S,T,X,Y)$-colouring if and only if $G$ has a monochromatic perfect red-blue $(S^*,T^*,X^*,Y^*)$-colouring.
We can now apply Lemma~\ref{l-mono} to find in polynomial time whether or not $G$ has a monochromatic perfect red-blue $(S^*,T^*,X^*,Y^*)$-colouring.
The correctness of our algorithm follows from the above arguments.
As we branch $O(n)$ times and each branch takes polynomial time to process, the total running time of our algorithm is polynomial.
\end{proof}

\noindent
We now consider $P_6$-free graphs. As a consequence of Theorem~\ref{t-hp}, a $P_6$-free graph either has a small domination number, in which case we use Lemma~\ref{l-dom}, a monochromatic dominating set, in which case we use Lemma~\ref{l-dom2}, or it has radius~$2$, in which case we use Theorem~\ref{t-radius}.

\begin{theorem}\label{t-p6}
{\sc Perfect Matching Cut} is polynomial-time solvable for $P_6$-free graphs.
\end{theorem}

\begin{proof}
Let $G$ be a connected $P_6$-free graph. By Theorem~\ref{t-hp}, we find that $G$ has a  dominating induced $C_6$ or 
a dominating (not necessarily induced) complete bipartite graph $K_{r,s}$. By Observation~\ref{o} it suffices to check if $G$ has a perfect red-blue colouring.

If $G$ has a dominating induced $C_6$, then $G$ has domination number at most~$6$. In that case we apply Lemma~\ref{l-dom} to find in polynomial time if $G$ has a perfect red-blue colouring.
Suppose that $G$ has a dominating complete bipartite graph $H$ with partition classes $\{u_1,\ldots,u_r\}$ and
$\{v_1,\ldots,v_s\}$. We may assume without loss of generality that $r\leq s$.

If $r\geq 2$ and $s\geq 3$, then any starting pair $(\{u_i\},\{v_j\})$ yields a no-answer. Hence, $V(H)$ is monochromatic for any perfect red-blue colouring of $G$. 
This means that we can check in polynomial time by Lemma~\ref{l-dom2} if $G$ has a perfect red-blue colouring.

Now assume that $r=1$ or $s\leq 2$. In the first case, $G$ has a (not necessarily induced) dominating star and thus $G$ has radius~2, and we apply Theorem~\ref{t-radius}. In the second case, $r\leq s\leq 2$, and thus
$G$ has domination number at most~$4$, and we apply Lemma~\ref{l-dom}. Hence, in both cases, we find in polynomial time whether or not $G$ has a perfect red-blue colouring.
\end{proof}

\noindent
For our last result we again use Lemma~\ref{l-mono}.

\begin{theorem}\label{t-h}
Let $H$ be a graph. If {\sc Perfect Matching Cut} is polynomial-time solvable for $H$-free graphs, then it is so for $(H+P_4)$-free graphs.
\end{theorem}

\begin{proof}
Assume that {\sc Perfect Matching Cut} can be solved in polynomial time for $H$-free graphs. Let $G$ be a connected $(H+P_4)$-free graph. Say, $G$ has an induced subgraph $G'$ that is isomorphic to $H$; else we are done by our assumption. Let $G^*$ be the graph obtained from $G$ after removing every vertex that belongs to $G'$ or that has a neighbour in~$G'$. As $G'$ is isomorphic to $H$ and $G$ is $(H+P_4)$-free, we find that $G^*$ is $P_4$-free.

We use Observation~\ref{o}-(iii) and search for a perfect red-blue colouring. We define $n =|V(G)|$,  $m= |E(G)|$. Following our approach, we need a starting pair $(S',T')$ with core $(S'',T'')$. By definition, $|S''|=|T''|\geq 1$. Hence, we consider all $O(m)$ options of choosing an edge $uv$ from $E(G)$, one of whose end-vertices we colour red (say $u$, so $u\in S''$) and the other one blue (say $v$, so $v\in T''$). Afterwards, for each (uncoloured) vertex in $G'$ we consider all options of colouring it either red or blue. As $G'$ is isomorphic to $H$, the number of distinct options is a constant, namely $2^{|V(H)|}$. Now, for every red (blue) vertex of $G'$ with no blue (red) neighbour, we consider all $O(n)$ options of colouring exactly one of its neighbours blue (red).
Hence, afterwards each vertex of $V(G')\cup N(V(G'))$ is either coloured red or blue. This leads to  $O(m2^{|V(H)|}n^{|V(H)|})$ options (branches), which we handle one by one. 

Consider an option as described above. Let $S'$ consist of $u$ and all red vertices of $V(G')\cup N(V(G'))$, and let $T'$ consist of $v$ and all blue vertices of $V(G')\cup N(V(G'))$. In this way we obtain a starting pair $(S',T')$ with core $(S'',T'')$. We apply rules R1-R7 exhaustively. If we find a no-answer, then we can discard the option by Lemma~\ref{l-iff}. Else we found in polynomial time an intermediate $4$-tuple $(S,T,X,Y)$, such that  $G$ has a perfect red-blue $(S'',T'',S',T')$-colouring  if and only if $G$ has a perfect red-blue $(S,T,X,Y)$-colouring. 

Consider a connected component $F$ of  $G-(X\cup Y)$, for which the following holds:

\begin{enumerate}
\item $F$ contains two distinct vertices $u$ and $v$, each with no neighbours in $X\cup Y$ and moreover, $v$ is dominating $F$; and
\item every vertex in $F-\{u,v\}$ has a neighbour in both $X$ and $Y$.
\end{enumerate}

\noindent
As $G$ is connected, the fact that $u$ and $v$ have no neighbours in $X\cup Y$ implies that $F-\{u,v\}$ is non-empty. Every vertex in $F-\{u,v\}$ has a neighbour in both $X$ and $Y$ and thus their matching neighbour is not in $F$. Hence, all vertices of $F-\{u\}$ are coloured alike in every perfect red-blue $(S,T,X,Y)$-colouring of $G$. Hence, we can safely remove $u$ and $v$. Then, after finding a monochromatic perfect red-blue $(S,T,X,Y)$-colouring of $G-\{u,v\}$ we give $v$ the same colour as the vertices of $F-\{u,v\}$ and $u$ the opposite colour. 

We also perform this operation for all other connected components of $G-(X\cup Y)$ that have the above two properties.
This yields, in polynomial time, a new but equivalent problem instance, which we denote by $G$ again.

\begin{claim*} Every perfect red-blue $(S,T,X,Y)$-colouring of $G$ is monochromatic.
\end{claim*}

\begin{claimproof}
In order to see this claim, let $F$ be a connected component of $G-(X\cup Y)$. If $|V(F)|=1$, then
$F$ will be monochromatic. Assume $|V(F)|\geq 2$. As  $V(G')\cup N(V(G'))\subseteq S'\cup T'$ and $S'\subseteq X$ and $T'\subseteq Y$, we find that $V(F)$ belongs to $G^*$. Since $G^*$ is $P_4$-free, $F$ is $P_4$-free. It is well-known (see e.g. Lemma~2 in~\cite{KP20}) that every connected $P_4$-free graph has a spanning complete bipartite subgraph $K$. Say, $K$ is isomorphic to $K_{k,\ell}$ for some integers $1\leq k \leq \ell$.

If $k\geq 2$ and $\ell \geq 3$, then $F$ must be monochromatic. Now suppose that $k=\ell=2$, so $F$ contains a $C_4$ as spanning subgraph. If $K$ contains a vertex~$u$ that has a neighbour in both $X$ and $Y$, then the matching neighbour of $u$ is in $X\cup Y$, so not in $F$. Hence, the neighbours of $u$ in $F$ must receive the same colour as $u$, which means that the fourth vertex of $F$ must also receive the same colour as $u$ (if that vertex is not adjacent to $u$, then it will be adjacent to the two neighbours of $u$ in $F$, as $F$ contains a spanning $C_4$). So $F$ is monochromatic. 

We conclude that every vertex of $F$ is adjacent to at most one vertex of $X\cup Y$. As $G$ is connected, $F$ has at least one vertex~$v$ with a neighbour~$w$ in $X\cup Y$, say $w\in X$. Then the other three vertices of $F$ must also have a neighbour in~$X$ (and thus no neighbour in $Y$),  else we would have applied R6. The only way we can extend the red-blue $(S,T,X,Y)$-colouring to a perfect red-blue colouring of $G$ is by colouring each vertex of $F$ blue, so $F$ is monochromatic.

It remains to consider the case where $k=1$ and $\ell\geq 1$. In this case $F$ contains a vertex~$v$ such that $\{v\}$ dominates $F$. Then, every vertex in $F-v$ has either no neighbours in $X\cup Y$ or a neighbour in both $X$ and $Y$; else we would have applied R7. Let $U$ be the set of vertices in $F-v$ with no neighbour in $X\cup Y$. As $\{v\}$ dominates $F$, every connected component of $F-v$ is monochromatic. So, $v$ is the matching neighbour of every vertex of $U$. If $|U|\geq 2$, then $G$ has no perfect red-blue $(S,T,X,Y)$-colouring so the claim is true. If $|U|=0$, then the vertices in $F-v$ all have a neighbour both in $X$ and $Y$. So, they do not have their matching neighbour in $F$ and thus will receive the same colour as~$v$. Hence, $F$ is monochromatic. Assume $U=\{u\}$.

 If $v$ is adjacent to a vertex in $X$ and to a vertex in $Y$, then its matching neighbour is in $X\cup Y$, so not in $F$. As $\{v\}$ dominates $F$, this means that $F$ must be monochromatic. Hence, $v$ is adjacent to at most one vertex of $X\cup Y$.

Note that by construction, $v$ is adjacent to exactly one vertex~$w$ of $X\cup Y$. Then $u$ has at least one neighbour in $F-v$; else we would have applied R5. Let $u'$ be an arbitrary neighbour of $u$ in $F-v$. As both $u$ and $u'$ are adjacent to $v$,  it follows that $u,u',v$ are coloured alike. Hence, $u$ has no matching neighbour. This means that $G$ has no perfect red-blue $(S,T,X,Y)$-colouring and the claim is true.
\end{claimproof}

\medskip
\noindent
We now apply rules R1--R11 exhaustively. This takes polynomial time by Lemma~\ref{l-iff2}. If this leads to a no-answer, then by Lemma~\ref{l-iff2} we may discard the option. Suppose we obtain a final $4$-tuple $(S^*,T^*,X^*,Y^*)$. By Lemma~\ref{l-iff2}, 
$G$ has a monochromatic perfect red-blue $(S,T,X,Y)$-colouring if and only if $G$ has a monochromatic perfect red-blue $(S^*,T^*,X^*,Y^*)$-colouring.
We apply Lemma~\ref{l-mono} to find in polynomial time if the latter holds.  If so, we are done by the Claim, else we discard the option.

The correctness of our algorithm follows from its description. As the total number of branches is $O(m2^{|V(H)|}n^{|V(H)|})$ and we can process each branch in polynomial time, the total running time of our algorithm is polynomial. Hence, we have proven the theorem.
\end{proof}

\section{Conclusions}\label{s-con}

We found new results on $H$-free graphs for three closely related edge cut problems: the classical {\sc Matching Cut} problem and its variants, {\sc Disconnected Perfect Matching} and {\sc Perfect Matching Cut}. 
We summarized  all known and new results for $H$-free graphs in Theorems~\ref{t-main1}--\ref{t-main3}. 
Due to our systematic study we are now able to identify some interesting open questions.

First, is there a graph $H$ for which the problems behave differently on $H$-free graphs? The graph $H=4P_5$ is a potential candidate if one can generalize Theorem~\ref{t-h}. 
Also, does there exist, just as for the other two problems, a constant~$r$ such that {\sc Perfect Matching Cut} is \NP-complete for $P_r$-free graphs? Moreover, are {\sc Matching Cut} and  {\sc Disconnected Perfect Matching}, just like {\sc Perfect Matching Cut}, \NP-complete for graphs of girth~$g$, also if $g\geq 6$ (cf.~\cite{LL19} and~\cite{BP}, respectively)? In particular, if this holds for {\sc Disconnected Perfect Matching}, then {\sc Disconnected Perfect Matching} would be \NP-complete for $H$-free graphs whenever $H$ has a cycle, just like the other two problems.

First, as can be noticed from Theorems~\ref{t-main1}--\ref{t-main3}, our knowledge on the complexity of the three problems is different. In particular, does there exist a constant~$r$ such that {\sc Perfect Matching Cut} is \NP-complete for $P_r$-free graphs? For the other two problems such a constant exists. For {\sc Matching Cut} we improved the previous value $r=27$~\cite{Fe23} to $r=19$ and for {\sc Disconnected Perfect Matching} we showed that we can take
 $r=23$, addressing a question in~\cite{BP}. We expect that these values of $r$ might not be tight, but it does not seem straightforward to improve our current hardness constructions. 

Finally, let {\sc Maximum Matching Cut} be the problem of finding a matching cut with a maximum number of edges. Recall that {\sc Perfect Matching Cut} is the ``extreme'' variant of this problem, namely when we search for a matching cut with a maximum number of edges. Does there exist a graph $H$ such that {\sc Maximum Matching Cut} and {\sc Perfect Matching Cut} differ in complexity when restricted to $H$-free graphs? What is the complexity of {\sc Maximum Matching Cut} on $P_r$-free graphs?


\end{document}

%% file: intro-variants2.tex
\begin{tikzpicture}
\tikzstyle{cutedge}=[black, ultra thick]
\tikzstyle{matchedge}=[black, ultra thick, dotted]

\begin{scope}[shift={(10,0)}, scale=0.75]
\node[rvertex](p1) at (0,0){};
\node[bvertex](p2) at (1,0){};
\node[bvertex](p3) at (2,0){};
\node[rvertex](p4) at (3,0){};
\node[rvertex](p5) at (4,0){};
\node[bvertex](p6) at (5,0){};

\draw[cutedge](p1)--(p2);
\draw[hedge](p2)--(p3);
\draw[cutedge](p3)--(p4);
\draw[hedge](p4)--(p5);
\draw[cutedge](p5)--(p6);
\end{scope}

\begin{scope}[shift={(5,0)}, scale=0.75]
\node[rvertex](p1) at (0,0){};
\node[bvertex](p2) at (1,0){};
\node[bvertex](p3) at (2,0){};
\node[bvertex](p4) at (3,0){};
\node[bvertex](p5) at (4,0){};
\node[rvertex](p6) at (5,0){};

\draw[cutedge](p1)--(p2);
\draw[hedge](p2)--(p3);
\draw[hedge](p3)--(p4);
\draw[hedge](p4)--(p5);
\draw[cutedge](p5)--(p6);
\end{scope}

\begin{scope}[shift={(0,0)}, scale=0.75]
\node[rvertex](p1) at (0,0){};
\node[bvertex](p2) at (1,0){};
\node[bvertex](p3) at (2,0){};
\node[bvertex](p4) at (3,0){};
\node[rvertex](p5) at (4,0){};
\node[rvertex](p6) at (5,0){};

\draw[cutedge](p1)--(p2);
\draw[hedge](p2)--(p3);
\draw[hedge](p3)--(p4);
\draw[cutedge](p4)--(p5);
\draw[hedge](p5)--(p6);
\end{scope}

\end{tikzpicture}

%% file: Moshi.tex
\begin{tikzpicture}
\tikzstyle{bvertex}=[thin,circle,inner sep=0.cm, minimum size=1.7mm, fill=black, draw=black]
\tikzstyle{edge} = [thin, gray]

\def\k{0.7}
\node[bvertex, label= below:$u$](u) at (0,0){};
\node[bvertex, label= below:$v$](v) at (\k,0){};

\node[bvertex](u1) at (150:\k){};
\node[bvertex](u2) at (180:\k){};
\node[bvertex](u3) at (210:\k){};

\begin{scope}[shift= {(\k,0)}]
\node[bvertex](v1) at (45:\k){};
\node[bvertex](v2) at (15:\k){};
\node[bvertex](v3) at (-15:\k){};
\node[bvertex](v4) at (-45:\k){};
\end{scope}

\draw[thick, nicered](u)--(v);

\foreach \i in {1,...,3}{
	\draw[edge] (u) -- (u\i);
	}
	\foreach \i in {1,...,4}{
	\draw[edge] (v) -- (v\i);
	}

\begin{scope}[shift = {(5*\k,0)}]
\node[bvertex, label= below:$u$](u) at (0,0){};
\node[bvertex, label= below:$v$](v) at (1.41*\k,0){};
\node[bvertex, label= above:$w_1$](w1) at (0.705*\k,0.65*\k){};
\node[bvertex, label= below:$w_2$](w2) at (0.705*\k,-0.65*\k){};

\node[bvertex](u1) at (150:\k){};
\node[bvertex](u2) at (180:\k){};
\node[bvertex](u3) at (210:\k){};

\begin{scope}[shift= {(1.41*\k,0)}]
\node[bvertex](v1) at (45:\k){};
\node[bvertex](v2) at (15:\k){};
\node[bvertex](v3) at (-15:\k){};
\node[bvertex](v4) at (-45:\k){};
\end{scope}

\draw[thick, nicered](u)--(w1);
\draw[thick, nicered](u)--(w2);
\draw[thick, nicered](w1)--(v);
\draw[thick, nicered](w2)--(v);

\foreach \i in {1,...,3}{
	\draw[edge] (u) -- (u\i);
	}
	\foreach \i in {1,...,4}{
	\draw[edge] (v) -- (v\i);
	}
\end{scope}

\end{tikzpicture}

%% file: p17-const1.tex
\begin{tikzpicture}
	\node[vertex, label={[label distance = -0.2cm]above right:\strut $v_{x_i}^1$}](v1) at ( 1,1){};
	\node[vertex, label={[label distance = -0.2cm]above left:\strut $v_{x_i}^2$}](v2) at ( 0,1){};
	\node[vertex, label={[label distance = -0.2cm]above left:\strut $v_{x_i}^3$}](v3) at ( 0,0){};
	\node[vertex, label={[label distance = -0.2cm]above right: \strut $v_{x_i}^s$}](v4) at ( 1,0){};
	\foreach \i in {1,...,4}{
			\foreach \j in {1,...,\i}{
			\draw[hedge](v\i)--(v\j);
			}
		}
		
		\begin{scope}[shift = {(3,0)}]
	\node[vertex, label={[label distance = -0.2cm]above right:\strut $u_{x_i}^2$}](u1) at ( 1,1){};
	\node[vertex, label={[label distance = -0.2cm]above left:\strut $u_{x_i}^1$}](u2) at ( 0,1){};
	\node[vertex, label={[label distance = -0.2cm]above left:\strut $u_{x_i}^s$}](u3) at ( 0,0){};
	\node[vertex, label={[label distance = -0.2cm]above right:\strut $u_{x_i}^3$}](u4) at ( 1,0){};
	\foreach \i in {1,...,4}{
			\foreach \j in {1,...,\i}{
			\draw[hedge](u\i)--(u\j);
			}
		}
		\end{scope}
		
		\draw[hedge](0.5,-2) circle (0.7);
		\draw[hedge](3.5,-2) circle (0.7);
		
		\node[vertex, label={[label distance = -0.2cm]170:\strut $s_1^i$}](s1) at ( 0.5,-1.3){};
		\node[vertex, label={[label distance = -0.2cm]30:\strut $s_2^i$}](s2) at ( 3.5,-1.3){};
		
		\node[](S1) at ( 0.5,-2){$S_1$};
		\node[](S2) at ( 3.5,-2){$S_2$};
		
		\begin{scope}[shift={(0.5,-2)}]
		\foreach \i in {1,...,8}{
			\node[vertex](s1\i) at ( 360/8*\i: 0.7){};
		}
		\end{scope}
			\begin{scope}[shift={(3.5,-2)}]
		\foreach \i in {1,...,8}{
			\node[vertex](s1\i) at ( 360/8*\i: 0.7){};
		}
		\end{scope}
		
		\draw[hedge](s1)--(v4);
		\draw[hedge](s1)--(u3);
		\draw[hedge](s2)--(v4);
		\draw[hedge](s2)--(u3);		

\end{tikzpicture}

%% file: p17-const2.tex
\begin{tikzpicture}
\small
\node[vertex, label={[label distance = -0.1cm]above left:$v_{c_j}$}](v) at (0,0){};
\node[vertex](v1) at (90:1){};
\node[vertex](v2) at (180:1){};
\node[vertex](v3) at (270:1){};
\draw[hedge](v)--(v1);
\draw[hedge](v)--(v2);
\draw[hedge](v)--(v3);

\begin{scope}[shift={(90:1)},rotate = 45,scale=0.5]

\node[vertex](v11) at (1,0){};
\node[vertex](v12) at (1,1){};
\node[vertex](v13) at (0,1){};
\end{scope}

\draw[hedge](v1)--(v11);
\draw[hedge](v1)--(v12);
\draw[hedge](v1)--(v13);
\draw[hedge](v11)--(v12);
\draw[hedge](v11)--(v13);
\draw[hedge](v12)--(v13);

\begin{scope}[shift={(180:1)},rotate = 135, scale=0.5]

\node[vertex](v21) at (1,0){};
\node[vertex](v22) at (1,1){};
\node[vertex](v23) at (0,1){};
\end{scope}

\draw[hedge](v2)--(v21);
\draw[hedge](v2)--(v22);
\draw[hedge](v2)--(v23);
\draw[hedge](v21)--(v22);
\draw[hedge](v21)--(v23);
\draw[hedge](v22)--(v23);

\begin{scope}[shift={(270:1)}, rotate = 225,scale=0.5]

\node[vertex](v31) at (1,0){};
\node[vertex](v32) at (1,1){};
\node[vertex](v33) at (0,1){};
\end{scope}

\draw[hedge](v3)--(v31);
\draw[hedge](v3)--(v32);
\draw[hedge](v3)--(v33);
\draw[hedge](v31)--(v32);
\draw[hedge](v31)--(v33);
\draw[hedge](v32)--(v33);

\node[](a) at (-0.75,1.4) {$V_{x_{i_1}}$};
\node[](a) at (-2.15,0.1) {$V_{x_{i_2}}$};
\node[](a) at (-0.75,-1.3) {$V_{x_{i_3}}$};

\begin{scope}[shift={(2,0.75)}, yscale=0.7, xscale=0.5, rotate = 270]
\node[vertex, label=above:$u_{c_j}^1$](uc1) at (0,0){};
\node[vertex, label=below:$u_{c_j}^2$](uc2) at (2,0){};
\footnotesize

\draw[hedge](v)--(uc1);
\draw[hedge](v)--(uc2);

\node[vertex,label={[label distance=-0.0cm]above:$a_{c_j}^1$}](ac1) at (-1,2){};
\node[vertex,label={[label distance=-0.17cm]above left:$a_{c_j}^4$}](ac2) at (0,2){};
\node[vertex,label={[label distance=-0.05cm]right:$a_{c_j}^3$}](ac3) at (2,2){};
\node[vertex,label={[label distance=-0.0cm]below:$a_{c_j}^6$}](ac4) at (3,2){};
\node[vertex,label={[label distance=-0.1cm]above:$a_{c_j}^5$}](ac6) at (1.5,2){};

\node[vertex](ac5) at (0.5,2){};
\node[](labelac5) at (0.4,2.8){$a_{c_j}^2$};

\node[vertex](u1) at (-2,4){};
\node[vertex](u2) at (4,4){};
\node[vertex](u3) at (1,4){};

\draw[hedge](uc1)--(uc2);

\draw[hedge](uc1)--(ac1);
\draw[hedge](uc1)--(ac3);
\draw[hedge](uc1)--(ac5);
\draw[hedge](uc2)--(ac2);
\draw[hedge](uc2)--(ac4);
\draw[hedge](uc2)--(ac6);

\draw[hedge](ac1)--(u1);
\draw[hedge](ac2)--(u1);
\draw[hedge](ac3)--(u2);
\draw[hedge](ac4)--(u2);
\draw[hedge](ac5)--(u3);
\draw[hedge](ac6)--(u3);

\begin{scope}[shift={(-2,4)}, xscale = 5/7, rotate = 45]
\node[vertex](u11) at (1,0){};
\node[vertex](u12) at (1,1){};
\node[vertex](u13) at (0,1){};
\end{scope}

\draw[hedge](u1)--(u11);
\draw[hedge](u1)--(u12);
\draw[hedge](u1)--(u13);
\draw[hedge](u11)--(u12);
\draw[hedge](u11)--(u13);
\draw[hedge](u12)--(u13);

\begin{scope}[shift={(4,4)},xscale = 5/7,rotate = 45]
\node[vertex](u21) at (1,0){};
\node[vertex](u22) at (1,1){};
\node[vertex](u23) at (0,1){};
\end{scope}

\draw[hedge](u2)--(u21);
\draw[hedge](u2)--(u22);
\draw[hedge](u2)--(u23);
\draw[hedge](u21)--(u22);
\draw[hedge](u21)--(u23);
\draw[hedge](u22)--(u23);

\begin{scope}[shift={(1,4)}, xscale = 5/7,rotate = 45]
\node[vertex](u31) at (1,0){};
\node[vertex](u32) at (1,1){};
\node[vertex](u33) at (0,1){};
\end{scope}

\draw[hedge](u3)--(u31);
\draw[hedge](u3)--(u32);
\draw[hedge](u3)--(u33);
\draw[hedge](u31)--(u32);
\draw[hedge](u31)--(u33);
\draw[hedge](u32)--(u33);

\node[](a) at (-2,6.5) {$U_{x_{i_1}}$};
\node[](a) at (1,6.5) {$U_{x_{i_2}}$};
\node[](a) at (4,6.5) {$U_{x_{i_3}}$};

\end{scope}
\end{tikzpicture}

%% file: p17-colouring.tex
\begin{tikzpicture}

\small
\node[bvertex, label={[label distance = -0.15cm]above left:$v_{c_j}$}](v) at (0,0){};
\node[bvertex](v1) at (90:1){};
\node[bvertex](v2) at (180:1){};
\node[rvertex](v3) at (270:1){};
\draw[hedge](v)--(v1);
\draw[hedge](v)--(v2);
\draw[hedge](v)--(v3);

\begin{scope}[shift={(2,0.75)},yscale=0.7, xscale=0.5, rotate = 270]
\node[bvertex, label=above:$u_{c_j}^1$](uc1) at (0,0){};
\node[bvertex, label=below:$u_{c_j}^2$](uc2) at (2,0){};

\draw[hedge](v)--(uc1);
\draw[hedge](v)--(uc2);

\node[bvertex](ac1) at (-1,2){};
\node[bvertex](ac2) at (0,2){};
\node[rvertex](ac3) at (2,2){};
\node[bvertex](ac4) at (3,2){};
\node[bvertex](ac5) at (0.5,2){};
\node[rvertex](ac6) at (1.5,2){};

\node[bvertex](u1) at (-2,4){};
\node[rvertex](u2) at (4,4){};
\node[rvertex](u3) at (1,4){};

\draw[hedge](uc1)--(uc2);

\draw[hedge](uc1)--(ac1);
\draw[hedge](uc1)--(ac3);
\draw[hedge](uc1)--(ac5);
\draw[hedge](uc2)--(ac2);
\draw[hedge](uc2)--(ac4);
\draw[hedge](uc2)--(ac6);

\draw[hedge](ac1)--(u1);
\draw[hedge](ac2)--(u1);
\draw[hedge](ac3)--(u2);
\draw[hedge](ac4)--(u2);
\draw[hedge](ac5)--(u3);
\draw[hedge](ac6)--(u3);
\end{scope}

\end{tikzpicture}

%% file: p17-path.tex
\begin{tikzpicture}
\begin{scope}[yscale=0.5, xscale=0.7]
\node[vertex, label=below:$u_{c_j}^1$](uc1) at (0,0){};
\node[vertex, label=below:$u_{c_j}^2$](uc2) at (2,0){};

\node[vertex](ac1) at (-1,2){};
\node[vertex](ac2) at (0,2){};
\node[vertex](ac3) at (2,2){};
\node[vertex](ac4) at (3,2){};
\node[vertex](ac5) at (0.5,2){};
\node[vertex](ac6) at (1.5,2){};

\node[vertex](u1) at (-2,4){};
\node[vertex](u2) at (4,4){};
\node[vertex](u3) at (1,4){};

\draw[gedge](uc1)--(uc2);

\draw[gedge](uc1)--(ac1);
\draw[hedge](uc1)--(ac3);
\draw[hedge](uc1)--(ac5);
\draw[hedge](uc2)--(ac2);
\draw[gedge](uc2)--(ac4);
\draw[hedge](uc2)--(ac6);

\draw[gedge](ac1)--(u1);
\draw[hedge](ac2)--(u1);
\draw[hedge](ac3)--(u2);
\draw[gedge](ac4)--(u2);
\draw[hedge](ac5)--(u3);
\draw[hedge](ac6)--(u3);
\end{scope}

\end{tikzpicture}

%% file: p23-const1.tex
\begin{tikzpicture}
\begin{scope}[scale=0.75]
\small
	\begin{scope}[shift={(0,0)}, scale = 0.85]
	\node[vertex, label= {[label distance=-0.25cm] above right:$v_{x_i}^1$}](v1) at ( 360/7:1){};
	\node[vertex, label= {[label distance=-0.1cm] above:$v_{x_i}^2$}](v2) at ( 2*360/7:1){};
	\node[vertex,label={[label distance=-0.2cm] above left: $v_{x_i}^3$}](v3) at ( 3*360/7:1){};
	\node[vertex, label={[label distance=-0.17cm] left: $v_{x_i}^4$}](v4) at ( 4*360/7:1){};
	\node[vertex, label={[label distance=-0.3cm] below left: $v_{x_i}^5$}](v5) at ( 5*360/7:1){};
	\node[vertex, label={[label distance=-0.0cm] right: $v_{x_i}^6$}](v6) at ( 6*360/7:1){};
	\node[vertex, label={[label distance=-0.25cm] above right: $v_{x_i}^s$}](v7) at ( 0:1){};
	\foreach \i in {1,...,7}{
			\foreach \j in {1,...,\i}{
			\draw[hedge](v\i)--(v\j);
			}
		}
		\end{scope}
		
		\begin{scope}[shift = {(3.4,0)}, scale = 0.85]
	\node[vertex, label={[label distance=-0.1cm]above:\strut $u_{x_i}^2$}](u1) at ( 360/7+360/14:1){};
	\node[vertex, label={[label distance=-0.2cm]above left:\strut $u_{x_i}^1$}](u2) at ( 2*360/7+360/14:1){};
	\node[vertex, label={[label distance=-0.2cm]above left:\strut $u_{x_i}^s$}](u3) at ( 3*360/7+360/14:1){};
	\node[vertex, label={[label distance=-0.0cm]left:\strut $u_{x_i}^6$}](u4) at ( 4*360/7+360/14:1){};
	\node[vertex, label={[label distance=-0.25cm]below right:\strut $u_{x_i}^5$}](u5) at ( 5*360/7+360/14:1){};
	\node[vertex, label={[label distance=-0.1cm]right:\strut $u_{x_i}^4$}](u6) at ( 6*360/7+360/14:1){};
	\node[vertex, label={[label distance=-0.25cm]above right:\strut $u_{x_i}^3$}](u7) at ( 0+360/14:1){};
	\foreach \i in {1,...,7}{
			\foreach \j in {1,...,\i}{
			\draw[hedge](u\i)--(u\j);
			}
		}
		\end{scope}
		\end{scope}

		\begin{scope}[scale = 0.85]
		\draw[hedge](0.5,-2.25) circle (0.7);
		\draw[hedge](2.5,-2.25) circle (0.7);
		
		\node[vertex, label={[label distance=-0.2cm]above left:\strut $s_1^i$}](s1) at ( 0.5,-1.55){};
		\node[vertex, label={[label distance=-0.2cm]above right:\strut $s_2^i$}](s2) at ( 2.5,-1.55){};
		
		\node[](S1) at ( 0.5,-2.25){$S_1$};
		\node[](S2) at ( 2.5,-2.25){$S_2$};
		
		\begin{scope}[shift={(0.5,-2.25)}]
		\foreach \i in {1,...,8}{
			\node[vertex](s1\i) at ( 360/8*\i: 0.7){};
		}
		\end{scope}
			\begin{scope}[shift={(2.5,-2.25)}]
		\foreach \i in {1,...,8}{
			\node[vertex](s1\i) at ( 360/8*\i: 0.7){};
		}
		\end{scope}
		
		\draw[hedge](s1)--(v7);
		\draw[hedge](s1)--(u3);
		\draw[hedge](s2)--(v7);
		\draw[hedge](s2)--(u3);	
		\end{scope}	

\end{tikzpicture}

%% file: p23-const2.tex
\begin{tikzpicture}
\small

\node[vertex, label=below:$v_{c_j}^1$](v1) at (0,0){};
\node[vertex, label=below:$v_{c_j}^2$](v2) at (2,0){};

\node[vertex, ](vx1) at (-1.5,2){};
\node[vertex,](vx2) at (-0.5,2){};
\node[vertex,](vx3) at (0.5,2){};
\node[vertex](vx4) at (1.5,2){};
\node[vertex](vx5) at (2.5,2){};
\node[vertex](vx6) at (3.5,2){};

\draw[hedge](v1)--(v2);
\draw[hedge](v1)--(vx1);
\draw[hedge](v1)--(vx3);
\draw[hedge](v1)--(vx5);
\draw[hedge](v2)--(vx2);
\draw[hedge](v2)--(vx4);
\draw[hedge](v2)--(vx6);
\draw[hedge](vx1)--(vx2);
\draw[hedge](vx3)--(vx4);
\draw[hedge](vx5)--(vx6);

\begin{scope}[shift= {(-1,2.5)}, scale = 0.7]
\node[](a) at (0,1.5) {$V_{x_{i_1}}$};
\node[vertex](v11) at (180:1){};
\node[vertex](v12) at (135:1){};
\node[vertex](v13) at (90:1){};
\node[vertex](v14) at (45:1){};
\node[vertex](v15) at (0:1){};
\end{scope}

\foreach \i in {1,...,5}{
	\draw[hedge](vx1)--(v1\i);
	\draw[hedge](vx2)--(v1\i);
	\foreach \j in {\i,...,5}{
		\draw[hedge](v1\i) -- (v1\j);
	}
	}
	
	\begin{scope}[shift= {(1,2.5)}, scale = 0.7]
	\node[](a) at (0,1.5) {$V_{x_{i_2}}$};
\node[vertex](v21) at (180:1){};
\node[vertex](v22) at (135:1){};
\node[vertex](v23) at (90:1){};
\node[vertex](v24) at (45:1){};
\node[vertex](v25) at (0:1){};
\end{scope}

\foreach \i in {1,...,5}{
	\draw[hedge](vx3)--(v2\i);
	\draw[hedge](vx4)--(v2\i);
	\foreach \j in {\i,...,5}{
		\draw[hedge](v2\i) -- (v2\j);
	}
	}
	
	\begin{scope}[shift= {(3,2.5)}, scale = 0.7]
	\node[](a) at (0,1.5) {$V_{x_{i_3}}$};
\node[vertex](v31) at (180:1){};
\node[vertex](v32) at (135:1){};
\node[vertex](v33) at (90:1){};
\node[vertex](v34) at (45:1){};
\node[vertex](v35) at (0:1){};
\end{scope}

\foreach \i in {1,...,5}{
	\draw[hedge](vx5)--(v3\i);
	\draw[hedge](vx6)--(v3\i);
	\foreach \j in {\i,...,5}{
		\draw[hedge](v3\i) -- (v3\j);
	}
	}

/////////////////////////////////////////////////////////////
//Right part of the figure starts here////
////////////////////////////////////////////////////////////

\begin{scope}[shift = {(7,1.75)},yscale=0.5, xscale=0.8]
\node[vertex, label={[label distance=-0.1cm]left:$u_{c_j}^1$}](uc1) at (0,0){};
\node[vertex, label={[label distance=-0.1cm]right:$u_{c_j}^2$}](uc2) at (2,0){};
\node[vertex, label={[label distance=-0.1cm]left:$u_{c_j}^3$}](uc3) at (0,-1.5){};
\node[vertex, label={[label distance=-0.1cm]right:$u_{c_j}^4$}](uc4) at (2,-1.5){};

\footnotesize
\node[vertex, label={[label distance=-0.05cm]left:$a_{c_j}^1$}](ac1) at (-1,2){};
\node[vertex, label={[label distance=-0.0cm]left:$a_{c_j}^4$}](ac4) at (0,2){};
\node[vertex, label={[label distance=-0.05cm]above:$a_{c_j}^3$}](ac3) at (2,2){};
\node[vertex, label={[label distance=-0.1cm]left:$a_{c_j}^6$}](ac6) at (3,2){};
\node[vertex, label={[label distance=-0.2cm]above left:$a_{c_j}^2$}](ac2) at (0.5,2){};
\node[vertex, label={[label distance=-0.1cm]left:$a_{c_j}^5$}](ac5) at (1.5,2){};

\begin{scope}[shift= {(0,0.5)}]
\node[vertex, label={[label distance=-0.05cm]left:$a_{c_j}^7$}](ac7) at (-1,-4){};
\node[vertex, label={[label distance=-0.0cm]left:$a_{c_j}^{10}$}](ac10) at (0,-4){};
\node[vertex, label={[label distance=-0.05cm]below:$a_{c_j}^9$}](ac9) at (2,-4){};
\node[vertex, label={[label distance=-0.05cm]left:$a_{c_j}^{12}$}](ac12) at (3,-4){};
\node[vertex, label={[label distance=-0.25cm]below left:$a_{c_j}^8$}](ac8) at (0.5,-4){};
\node[vertex, label={[label distance=-0.1cm]left:$a_{c_j}^{11}$}](ac11) at (1.5,-4){};
\end{scope}

\node[vertex](u1) at (-2,4){};
\node[vertex](u3) at (4,4){};
\node[vertex](u2) at (1,4){};

\begin{scope}[shift= {(0,0.5)}]
\node[vertex](u4) at (-2,-6){};
\node[vertex](u6) at (4,-6){};
\node[vertex](u5) at (1,-6){};
\end{scope}

\draw[hedge](uc1)--(uc2);
\draw[hedge](uc1)--(uc3);
\draw[hedge](uc1)--(uc4);
\draw[hedge](uc2)--(uc3);
\draw[hedge](uc2)--(uc4);
\draw[hedge](uc3)--(uc4);

\draw[hedge](uc1)--(ac1);
\draw[hedge](uc1)--(ac3);
\draw[hedge](uc1)--(ac2);
\draw[hedge](uc2)--(ac5);
\draw[hedge](uc2)--(ac4);
\draw[hedge](uc2)--(ac6);

\draw[hedge](ac1)--(u1);
\draw[hedge](ac4)--(u1);
\draw[hedge](ac2)--(u2);
\draw[hedge](ac5)--(u2);
\draw[hedge](ac3)--(u3);
\draw[hedge](ac6)--(u3);

\draw[hedge](uc3)--(ac7);
\draw[hedge](uc3)--(ac9);
\draw[hedge](uc3)--(ac8);
\draw[hedge](uc4)--(ac11);
\draw[hedge](uc4)--(ac10);
\draw[hedge](uc4)--(ac12);

\draw[hedge](ac7)--(u4);
\draw[hedge](ac10)--(u4);
\draw[hedge](ac8)--(u5);
\draw[hedge](ac11)--(u5);
\draw[hedge](ac9)--(u6);
\draw[hedge](ac12)--(u6);

\draw[hedge, ](ac1) to [bend right = 30](ac7);
\draw[hedge](ac2) to [bend right = 0](ac8);
\draw[hedge](ac3) to [bend left = 30](ac9);
\draw[hedge](ac4) to [bend right = 30](ac10);
\draw[hedge](ac5) to [bend left = 0](ac11);
\draw[hedge](ac6) to [bend left = 30](ac12);

//blueclique
\node[](a) at (-0.85,4.5) {$U_{x_{i_1}}$};
\node[vertex](x1) at (-2.5,4.5){};
\node[vertex](x2) at (-1.5,4.5){};

\begin{scope}[shift= {(0,0.5)}]
\node[](a) at (-0.85 ,-6.5) {$U_{x_{i_1}}$};
\node[vertex](x3) at (-2.5,-6.5){};
\node[vertex](x4) at (-1.5,-6.5){};
\end{scope}

\draw[bluedge](x1)--(u1);
\draw[bluedge](x2)--(u1);
\draw[bluedge](x1)--(x2);

\draw[bluedge](x1)--(-2.5,4.9);
\draw[bluedge](x1)--(-1.5,4.9);
\draw[bluedge](x2)--(-1.6,5);
\draw[bluedge](x2)--(-2.4,5);
\draw[bluedge](-2,5)--(u1);
\draw[bluedge](-1.6,4.9)--(u1);
\draw[bluedge](-2.4,4.9)--(u1);

\draw[bluedge](x3)--(u4);
\draw[bluedge](x4)--(u4);
\draw[bluedge](x3)--(x4);

\begin{scope}[shift= {(0,0.5)}]
\draw[bluedge](x3)--(-2.5,-6.9);
\draw[bluedge](x3)--(-1.6,-7);
\draw[bluedge](x4)--(-1.5,-6.9);
\draw[bluedge](x4)--(-2.4,-7);
\draw[bluedge](-2,-7)--(u4);
\draw[bluedge](-1.6,-6.9)--(u4);
\draw[bluedge](-2.4,-6.9)--(u4);
\end{scope}

//redclique
\node[](a) at (2.15,4.5) {$U_{x_{i_2}}$};
\node[vertex](y1) at (0.5,4.5){};
\node[vertex](y2) at (1.5,4.5){};

\begin{scope}[shift= {(0,0.5)}]
\node[](a) at (2.15 ,-6.5) {$U_{x_{i_2}}$};
\node[vertex](y3) at (0.5,-6.5){};
\node[vertex](y4) at (1.5,-6.5){};
\end{scope}

\draw[rededge](y1)--(u2);
\draw[rededge](y2)--(u2);
\draw[rededge](y1)--(y2);

\draw[rededge](y1)--(0.5,4.9);
\draw[rededge](y1)--(1.5,4.9);
\draw[rededge](y2)--(1.4,5);
\draw[rededge](y2)--(0.6,5);
\draw[rededge](1,5)--(u2);
\draw[rededge](1.4,4.9)--(u2);
\draw[rededge](0.6,4.9)--(u2);

\draw[rededge](y3)--(u5);
\draw[rededge](y4)--(u5);
\draw[rededge](y3)--(y4);

\begin{scope}[shift= {(0,0.5)}]
\draw[rededge](y3)--(0.5,-6.9);
\draw[rededge](y3)--(1.4,-7);
\draw[rededge](y4)--(1.5,-6.9);
\draw[rededge](y4)--(0.6,-7);
\draw[rededge](1,-7)--(u5);
\draw[rededge](1.4,-6.9)--(u5);
\draw[rededge](0.6,-6.9)--(u5);
\end{scope}

//greenclique
\node[](a) at (5.15,4.5) {$U_{x_{i_3}}$};
\node[vertex](z1) at (3.5,4.5){};
\node[vertex](z2) at (4.5,4.5){};

\begin{scope}[shift= {(0,0.5)}]
\node[](a) at (5.15 ,-6.5) {$U_{x_{i_3}}$};
\node[vertex](z3) at (3.5,-6.5){};
\node[vertex](z4) at (4.5,-6.5){};
\end{scope}

\draw[grnedge](z1)--(u3);
\draw[grnedge](z2)--(u3);
\draw[grnedge](z1)--(z2);

\draw[grnedge](z1)--(3.5,4.9);
\draw[grnedge](z1)--(4.5,4.9);
\draw[grnedge](z2)--(4.4,5);
\draw[grnedge](z2)--(3.6,5);
\draw[grnedge](4,5)--(u3);
\draw[grnedge](4.4,4.9)--(u3);
\draw[grnedge](3.6,4.9)--(u3);

\draw[grnedge](z3)--(u6);
\draw[grnedge](z4)--(u6);
\draw[grnedge](z3)--(z4);

\begin{scope}[shift= {(0,0.5)}]
\draw[grnedge](z3)--(3.5,-6.9);
\draw[grnedge](z3)--(4.4,-7);
\draw[grnedge](z4)--(4.5,-6.9);
\draw[grnedge](z4)--(3.6,-7);
\draw[grnedge](4,-7)--(u6);
\draw[grnedge](4.4,-6.9)--(u6);
\draw[grnedge](3.6,-6.9)--(u6);
\end{scope}

\end{scope}

\end{tikzpicture}

%% file: p23-perfect.tex
\begin{tikzpicture}
\small

\node[bvertex, label=below:$v_{c_j}^1$](v1) at (0,0){};
\node[bvertex, label=below:$v_{c_j}^2$](v2) at (2,0){};

\node[bvertex, ](vx1) at (-1.5,2){};
\node[bvertex,](vx2) at (-0.5,2){};
\node[bvertex,](vx3) at (0.5,2){};
\node[bvertex](vx4) at (1.5,2){};
\node[rvertex](vx5) at (2.5,2){};
\node[rvertex](vx6) at (3.5,2){};

\draw[hedge](v1)--(v2);
\draw[hedge](v1)--(vx1);
\draw[hedge](v1)--(vx3);
\draw[medge](v1)--(vx5);
\draw[hedge](v2)--(vx2);
\draw[hedge](v2)--(vx4);
\draw[medge](v2)--(vx6);
\draw[hedge](vx1)--(vx2);
\draw[hedge](vx3)--(vx4);
\draw[hedge](vx5)--(vx6);

\begin{scope}[yscale = 0.5,xscale=0.8, shift = {(8,3)}]

\node[bvertex, label={[label distance=-0.1cm]left:$u_{c_j}^1$}](uc1) at (0,0){};
\node[bvertex, label={[label distance=-0.1cm]right:$u_{c_j}^2$}](uc2) at (2,0){};
\node[bvertex, label={[label distance=-0.1cm]left:$u_{c_j}^3$}](uc3) at (0,-2){};
\node[bvertex, label={[label distance=-0.1cm]right:$u_{c_j}^4$}](uc4) at (2,-2){};

\node[bvertex](ac1) at (-1,2){};
\node[bvertex](ac4) at (0,2){};
\node[rvertex](ac3) at (2,2){};
\node[bvertex](ac6) at (3,2){};
\node[bvertex](ac2) at (0.5,2){};
\node[rvertex](ac5) at (1.5,2){};

\node[bvertex](ac7) at (-1,-4){};
\node[bvertex](ac10) at (0,-4){};
\node[rvertex](ac9) at (2,-4){};
\node[bvertex](ac12) at (3,-4){};
\node[bvertex](ac8) at (0.5,-4){};
\node[rvertex](ac11) at (1.5,-4){};

\node[bvertex](u1) at (-2,4){};
\node[rvertex](u3) at (4,4){};
\node[rvertex](u2) at (1,4){};

\node[bvertex](u4) at (-2,-6){};
\node[rvertex](u6) at (4,-6){};
\node[rvertex](u5) at (1,-6){};

\draw[hedge](uc1)--(uc2);
\draw[hedge](uc1)--(uc3);
\draw[hedge](uc1)--(uc4);
\draw[hedge](uc2)--(uc3);
\draw[hedge](uc2)--(uc4);
\draw[hedge](uc3)--(uc4);

\draw[hedge](uc1)--(ac1);
\draw[medge](uc1)--(ac3);
\draw[hedge](uc1)--(ac2);
\draw[medge](uc2)--(ac5);
\draw[hedge](uc2)--(ac4);
\draw[hedge](uc2)--(ac6);

\draw[hedge](ac1)--(u1);
\draw[hedge](ac4)--(u1);
\draw[medge](ac2)--(u2);
\draw[hedge](ac5)--(u2);
\draw[hedge](ac3)--(u3);
\draw[medge](ac6)--(u3);

\draw[hedge](uc3)--(ac7);
\draw[medge](uc3)--(ac9);
\draw[hedge](uc3)--(ac8);
\draw[medge](uc4)--(ac11);
\draw[hedge](uc4)--(ac10);
\draw[hedge](uc4)--(ac12);

\draw[hedge](ac7)--(u4);
\draw[hedge](ac10)--(u4);
\draw[medge](ac8)--(u5);
\draw[hedge](ac11)--(u5);
\draw[hedge](ac9)--(u6);
\draw[medge](ac12)--(u6);

\draw[medge, ](ac1) to [bend right = 30](ac7);
\draw[hedge](ac2) to [bend right = 0](ac8);
\draw[hedge](ac3) to [bend left = 30](ac9);
\draw[medge](ac4) to [bend right = 30](ac10);
\draw[hedge](ac5) to [bend left = 0](ac11);
\draw[hedge](ac6) to [bend left = 30](ac12);

\end{scope}

\end{tikzpicture}

%% file: p23-path.tex
\begin{tikzpicture}
\small
\begin{scope}[yscale = 0.5, xscale = 0.8]

\node[vertex, label={[label distance=-0.1cm]left:$u_{c_j}^1$}](uc1) at (0,0){};
\node[vertex, label={[label distance=-0.1cm]right:$u_{c_j}^2$}](uc2) at (2,0){};
\node[vertex, label={[label distance=-0.1cm]left:$u_{c_j}^3$}](uc3) at (0,-2){};
\node[vertex, label={[label distance=-0.1cm]right:$u_{c_j}^4$}](uc4) at (2,-2){};

\node[vertex](ac1) at (-1,2){};
\node[vertex](ac4) at (0,2){};
\node[vertex](ac3) at (2,2){};
\node[vertex](ac6) at (3,2){};
\node[vertex](ac2) at (0.5,2){};
\node[vertex](ac5) at (1.5,2){};

\node[vertex](ac7) at (-1,-4){};
\node[vertex](ac10) at (0,-4){};
\node[vertex](ac9) at (2,-4){};
\node[vertex](ac12) at (3,-4){};
\node[vertex](ac8) at (0.5,-4){};
\node[vertex](ac11) at (1.5,-4){};

\node[vertex](u1) at (-2,4){};
\node[vertex](u3) at (4,4){};
\node[vertex](u2) at (1,4){};

\node[vertex](u4) at (-2,-6){};
\node[vertex](u6) at (4,-6){};
\node[vertex](u5) at (1,-6){};

\draw[hedge](uc1)--(uc2);
\draw[hedge](uc1)--(uc3);
\draw[hedge](uc1)--(uc4);
\draw[hedge](uc2)--(uc3);
\draw[hedge](uc2)--(uc4);
\draw[gedge](uc3)--(uc4);

\draw[hedge](uc1)--(ac1);
\draw[hedge](uc1)--(ac3);
\draw[hedge](uc1)--(ac2);
\draw[hedge](uc2)--(ac5);
\draw[hedge](uc2)--(ac4);
\draw[hedge](uc2)--(ac6);

\draw[gedge](ac1)--(u1);
\draw[hedge](ac4)--(u1);
\draw[hedge](ac2)--(u2);
\draw[hedge](ac5)--(u2);
\draw[hedge](ac3)--(u3);
\draw[gedge](ac6)--(u3);

\draw[gedge](uc3)--(ac7);
\draw[hedge](uc3)--(ac9);
\draw[hedge](uc3)--(ac8);
\draw[hedge](uc4)--(ac11);
\draw[hedge](uc4)--(ac10);
\draw[gedge](uc4)--(ac12);

\draw[hedge](ac7)--(u4);
\draw[hedge](ac10)--(u4);
\draw[hedge](ac8)--(u5);
\draw[hedge](ac11)--(u5);
\draw[hedge](ac9)--(u6);
\draw[hedge](ac12)--(u6);

\draw[gedge, ](ac1) to [bend right = 30](ac7);
\draw[hedge](ac2) to [bend right = 0](ac8);
\draw[hedge](ac3) to [bend left = 30](ac9);
\draw[hedge](ac4) to [bend right = 30](ac10);
\draw[hedge](ac5) to [bend left = 0](ac11);
\draw[gedge](ac6) to [bend left = 30](ac12);
\end{scope}

\end{tikzpicture}

%% file: p23-comps.tex
\begin{tikzpicture}
\begin{scope}[scale=1, rotate = -39]
\small
	\begin{scope}[shift={(0,0)}]
	\node[vertex, ](v1) at ( 360/7:1){};
	\node[vertex, ](v2) at ( 2*360/7:1){};
	\node[vertex](v3) at ( 3*360/7:1){};
	\node[vertex](v4) at ( 4*360/7:1){};
	\node[vertex](v5) at ( 5*360/7:1){};
	\node[vertex](v6) at ( 6*360/7:1){};
	\node[vertex](v7) at ( 0:1){};
	\foreach \i in {1,...,7}{
			\foreach \j in {1,...,\i}{
			\draw[hedge](v\i)--(v\j);
			}
		}
		\end{scope}
		\node[vertex, ](u1) at ( 32.5:1.75){};
		\node[vertex, ](u2) at ( 17.5:1.75){};
		\node[vertex, ](w1) at ( 35:2.5){};
		\node[vertex, ](w2) at ( 15:2.5){};
		
		\node[vertex, ](u3) at ( 136.5:1.75){};
		\node[vertex, ](u4) at ( 121.5:1.75){};
		\node[vertex, ](w3) at ( 139:2.5){};
		\node[vertex, ](w4) at ( 119:2.5){};		
		
		\node[vertex, ](u5) at ( 240.5:1.75){};
		\node[vertex, ](u6) at ( 225.5:1.75){};
		\node[vertex, ](w5) at ( 243:2.5){};
		\node[vertex, ](w6) at ( 223:2.5){};		
		
		\draw[gedge](v7)--(v4);
		
	\draw[hedge](v1)--(u1);	
	\draw[gedge](v7)--(u2);	
	\draw[gedge](u1)--(u2);		
	\draw[hedge](v1)--(w1);	
	\draw[hedge](v7)--(w2);	
	\draw[hedge](w1)--(w2);
	
	\draw[hedge](v3)--(u3);	
	\draw[hedge](v2)--(u4);	
	\draw[hedge](u3)--(u4);		
	\draw[hedge](v3)--(w3);	
	\draw[hedge](v2)--(w4);	
	\draw[hedge](w3)--(w4);
			
	\draw[hedge](v5)--(u5);	
	\draw[hedge](v4)--(u6);	
	\draw[hedge](u5)--(u6);		
	\draw[hedge](v5)--(w5);	
	\draw[gedge](v4)--(w6);	
	\draw[gedge](w5)--(w6);	
		
\end{scope}

\end{tikzpicture}

%% file: lemma-rules-morespace.tex
\begin{tikzpicture}
\tikzstyle{evertex}=[thin,circle,inner sep=0.cm, minimum size=1.7mm, fill=none, draw=black]
\tikzstyle{rahmen1}=[rounded corners = 5pt,draw, dashed, minimum width = 15pt, minimum height = 30pt]
\tikzstyle{rahmen2}=[rounded corners = 5pt,draw, dashed, minimum width = 45pt, minimum height = 25pt]
\tikzstyle{rahmen3}=[rounded corners = 5pt,draw, dashed, minimum width = 140pt, minimum height = 30pt]
	
		\node[rvertex](v1) at (0,0){};
	\node[rvertex](v2) at (1,0){};
	\node[rvertex](v3) at (2,0){};
	\node[evertex](v4) at (2.7,0.5){};
	\node[evertex](v5) at (3.4,0.5){};
	\node[rvertex](v6) at (4.1,0){};
	\node[bvertex](v7) at (0,1){};
	\node[bvertex](v8) at (1,1){};
	\node[bvertex](v9) at (2,1){};
	\node[evertex](v10) at (2,0.5){};
	\node[evertex](v11) at (4.1,0.5){};
	\node[bvertex](v12) at (4.1,1){};
	
	\node[rahmen2](T) at (0.5,1.25){};
	\node[rahmen2](S) at (0.5,-0.25){};
	\node[](Tl) at (0.5, 1.4){$T$};
	\node[](Sl) at (0.5, -0.4){$S$};
	\node[rahmen3](S) at (2.05,-0.25){};
	\node[rahmen3](T) at (2.05,1.25){};
	\node[](Tl) at (4.1, 1.4){$Y$};
	\node[](Sl) at (4.1, -0.4){$X$};
	
	\draw[hedge](v1)--(v7);
	\draw[hedge](v7)--(v8);
	\draw[hedge](v2)--(v8);
	\draw[hedge](v2)--(v3);
	\draw[hedge](v8)--(v9);
	\draw[hedge](v3)--(v4);
	\draw[hedge](v3)--(v10);
	\draw[hedge](v4)--(v5);
	\draw[hedge](v5)--(v6);
	\draw[hedge](v5)--(v12);
	\draw[hedge](v6)--(v11);
	\draw[hedge](v4)--(v9);
	\draw[hedge](v9)--(v10);
	\draw[hedge](v11)--(v12);
		\path (v2) edge [out=340,in=200, thick, draw=gray] (v6);
		\path (v3) edge [out=340,in=200, thick, draw=gray] (v6);
		\path (v8) edge [out=20,in=160, thick, draw=gray] (v12);

\end{tikzpicture}